\documentclass[12pt]{amsart}
\usepackage {amscd}
\usepackage[german,english]{babel}
\usepackage {amsfonts}
\usepackage[all]{xy}
\usepackage{amssymb,amsmath, amsthm,latexsym,verbatim}
\usepackage{a4wide}
\usepackage[hypertex]{hyperref}\usepackage{enumitem}

\newcommand{\C}{\mathbb{C}}
\newcommand{\R}{\mathbb{R}}
\newcommand{\Z}{\mathbb{Z}}

\newcommand{\N}{\mathbb{N}}

\newcommand{\Cl}{\operatorname{Cl}}

\newcommand{\s}{\operatorname{s}}

\newcommand{\rk}{\operatorname{rk}}
\newcommand{\kL}{\mathfrak{k}}
\newcommand{\mL}{\mathfrak{m}}
\newcommand{\aL}{\mathfrak{a}}
\newcommand{\gL}{\mathfrak{g}}
\newcommand{\pL}{\mathfrak{p}}
\newcommand{\nf}{\mathfrak{n}}

\newcommand{\bL}{\mathfrak{b}}
\newcommand{\uu}{\operatorname{u}}

\newcommand{\parab}{\operatorname{par}}
\newcommand{\Id}{\operatorname{Id}}

\newcommand{\Spin}{\operatorname{Spin}}

\newcommand{\Hom}{\operatorname{Hom}}

\newcommand{\SU}{\operatorname{SU}}
\newcommand{\Ad}{\operatorname{Ad}}
\newcommand{\Sl}{\operatorname{SL}}
\newcommand{\SO}{\operatorname{SO}}
\newcommand{\Iim}{\operatorname{Im}}
\newcommand{\cc}{\operatorname{c}}
\newcommand{\CC}{\operatorname{C}}
\newcommand{\gr}{\operatorname{gr}}

\newcommand{\vol}{\operatorname{vol}}
\newcommand{\Tr}{\operatorname{Tr}}

\newcommand{\tr}{\operatorname{tr}}
\newcommand{\End}{\operatorname{End}}
\newcommand{\Real}{\operatorname{Re}}

\newtheorem {thrm}{Theorem}[section]
\newtheorem {prop}[thrm] {Proposition}
\newtheorem {lem}[thrm] {Lemma}
\newtheorem {kor}[thrm]{Corollary}
\theoremstyle{definition}

\theoremstyle{remark}
\newtheorem {bmrk}[thrm] {Remark}

\begin{document}

\author{Jonathan Pfaff}
\address{Universit\"at Bonn\\
Mathematisches Institut\\
Endenicher Alle 60\\
D -- 53115 Bonn, Germany}
\email{pfaff@math.uni-bonn.de}

\title[]{Selberg zeta functions on odd-dimensional hyperbolic manifolds of
finite
volume}

\begin{abstract}
We study Selberg zeta functions $Z(s,\sigma)$  
associated 
to locally homogeneous vector bundles over the unit-sphere bundle   
of a complete odd-dimensional hyperbolic manifold of finite volume. We assume a
certain condition on the fundamental group of the manifold.
A priori, the Selberg zeta functions are defined only for $s$ in some 
right half-space of $\C$. We will prove that for any locally homogeneous bundle
the functions $Z(s,\sigma)$ 
have a meromorphic continuation to $\C$ and we will give a complete 
description of their singularities in terms of spectral data of 
the underlying manifold. Our work generalizes results of 
Bunke and Olbrich to the non-compact situation. As an application of our results one can 
compare the normalized Reidemeister torsions on hyperbolic 3 manifolds with
cusps which were introduced by Menal-Ferrer 
and Porti to the corresponding regularized analytic torsions.
\end{abstract}

\maketitle
\setcounter{tocdepth}{1}
\section{Introduction}
\setcounter{equation}{0}
Let $X$ be a complete hyperbolic manifold of finite volume and of odd dimension
$d$. Then it is possible to encode geometric data of $X$ associated to its
geodesic flow
into dynamical zeta functions, called Selberg zeta functions. These 
functions are associated to locally homogeneous vector bundles over the unit
sphere bundle $SX$ of $X$. 
They are defined as an infinite product which converges only in 
some right half-space of the complex plane. The purpose 
of this paper is to prove a meromorphic continuation of the zeta functions 
associated to any locally homogeneous vector bundle to $\C$ and to relate
their singularities 
to spectral data of $X$. In the non-compact case, such a result had previously
been known only for a finite number
of special bundles. 

We shall now define the Selberg zeta functions more precisely. For further 
details we refer the reader to section \ref{secselz}. The tangent bundle $TSX$
of $SX$ is a
locally 
homogeneous vector bundle and we fix an invariant metric. Let $\Phi$ be
the geodesic flow of $X$ acting on $SX$.
Let $d\Phi$ denote the differential of $\Phi$ with respect to $SX$. Then $\Phi$
has the
Anosov
property. This means that there is a 
$d\Phi$-invariant
splitting 
\begin{align}\label{splitting}
TSX=T^{\s}SX \oplus T^{\uu} SX\oplus T^0 SX,
\end{align}
where $d\Phi$ is exponentially
shrinking on $T^{\s}SX$ and exponentially expanding on $T^{\uu}SX$ as 
$t\to\infty$ with respect to the metric and where $T^0 SX$ is the
one-dimensional subspace tangent to the
flow. 

Denote by 
$\mathcal{C}(X)$ the set of non-trivial closed geodesics on $X$
parametrized by arc-length and for 
$c\in\mathcal{C}(X)$ let $\ell(c)$ be its length. The set $\mathcal{C}(X)$ 
corresponds bijectively to the closed orbits of $\Phi$. A closed geodesic $c\in
\mathcal{C}(X)$ 
is called prime if it is the shortest among the closed geodesics having the 
same image as $c$.
The set of prime geodesics will be denoted by $\mathcal{PC}(X)$.
For $c\in\mathcal{C}(X)$ let $P_c:=d\Phi|_{(\ell(c),\dot{c}(0))}$ be its
monodromy
map, also 
called the Poincar\'e map of $c$. Then
$P_c$ respects the splitting \eqref{splitting} and its restriction to the fibre
of $T^{\s}
SX$ over $\dot{c}(0)$ will be denoted by $P_c^{\s}$.

Now we recall that $X$ can be realized as $X=\Gamma\backslash \widetilde X$,
where $\widetilde X=G/K$ with
$G=\Spin(d,1)$,
$K=\Spin(d)$ and 
where $\Gamma$ is a discrete, torsion-free subgroup of $G$. The space
$\widetilde X$ can be 
identified with the $d$-dimensional hyperbolic space. Let
$P_0:=MAN$ be the standard parabolic subgroup 
of $G$ and let $\sigma$ be a finite dimensional unitary representation of $M$.
Then $\sigma$ naturally defines a locally homogeneous vector bundle
$V(\sigma)$ over $SX$ and the
geodesic
flow lifts to a flow on $V(\sigma)$. Thus for every closed geodesic $c$  
its lift to $V(\sigma)$ defines an endomorphism 
$\mu_\sigma(c)$ on the fibre of
$V(\sigma)$ over $\dot{c}(0)$. Let $d=2n+1$, then the Selberg zeta function
$Z(s,\sigma)$ is defined as 
\begin{align}\label{SZF}
Z(s,\sigma):=\prod_{c\in\mathcal{PC}(X)}\prod_{k=0}^\infty\det{
\left(\Id-\mu_\sigma(c)\otimes
S^kP^{\s}_ce^{-(s+n)\ell(c)}
\right)}.
\end{align}
Here $S^k$ denotes the $k$-th symmetric power of an endomorphism. The infinite
product in \eqref{SZF} converges absolutely only for
$\Real(s)>2n$. In this paper we prove that  $Z(s,\sigma)$ has 
a meromorphic continuation to $\C$ and describe its 
singularities. 

From now
on we assume that $\Gamma$ satisifies the following 
condition: For every $\Gamma$-cuspidal parabolic subgroup  
$P=M_PA_PN_P$ of $G$ one has 
\begin{align}\label{a1}
\Gamma\cap P=\Gamma\cap N_P. 
\end{align}
The central elements of our main result can be summarized as follows.
\begin{thrm}\label{Erstes Theorem}
Let $X$ be a complete odd dimensional hyperbolic manifold of finite volume and
assume that
its
fundamental group $\Gamma$ satisfies \eqref{a1}. Let $\sigma$ be a finite
dimensional
unitary representation of $M$. Then the
Selberg zeta function $Z(s,\sigma)$ has a
meromorphic continuation
to $\C$. All its possible singularities (zeroes and poles) and their
corresponding orders  can be described in terms of
the following data: 
\begin{itemize}
\item by the discrete spectrum of a graded differential operator
$A(\sigma)$ of Laplace type and of a twisted Dirac operator $D(\sigma)$
which act on a locally homogeneous vector bundle $E(\sigma)$ over $X$.
\item by the poles of the scattering matrix
$\mathbf{C}(\nu_\sigma:\sigma:s)$ associated to
$\sigma$ and a certain
representation $\nu_\sigma$ of $K$.
\end{itemize}
Additionally, the Selberg zeta function has singularities which are 
independent of $X$ and whose order depends only on $p$, the number of 
cusps of $X$.
\end{thrm}
Theorem \ref{Erstes Theorem} provides a 
relation between the geometry of the possibly non-compact hyperbolic manifold
$X$
and the spectrum of certain differential operators. 
More precisely, the singularities in the first two items correspond
to spectral
parameters of $X$ since poles of the scattering matrix
$\mathbf{C}(\nu_\sigma:\sigma:s)$ are related to poles of the
resolvent of $A(\sigma)$. 

The additional singularities of $Z(s,\sigma)$ on the
negative
real line arise from the contribution of weighted orbital integrals to the
geometric
side of the trace formula. For a more precise version of Theorem \ref{Erstes
Theorem} we refer the reader
to Theorem \ref{AnconctZ} at the end of this article .\\

Our main result implies the existence of the analytic continuation of the
Ruelle 
zeta functions associated to representations of $M$ and $G$. These functions are
dynamical zeta functions, defined
as follows. Let $\sigma$ be a finite dimensional unitary representation of $M$.
Then put
\begin{align}\label{Ruelle 1}
R(s,\sigma):=\prod_{c\in\mathcal{PC}(X)}\det{\left(\Id-\mu_\sigma(c)e^{
-s\ell(
c)}\right)}.
\end{align}
Similarly, if $\tau$ is a finite dimensional representation  of $G$, 
the restriction of $\tau$ to $\Gamma$ defines 
a locally homogeneous vector bundle $V(\tau)$ over $SX$ and the geodesic 
flow lifts to a flow on $V(\tau)$. Thus every closed geodesic $c$ induces an 
endomorphism $\mu_\tau(c)$ on the fibre
of $V(\tau)$ over $\dot{c}(0)$. Now we define the Ruelle zeta function on $X$
associated to $\tau$ by
\begin{align}\label{Ruelle 2}
R(s,\tau):=\prod_{c\in\mathcal{PC}(X)}\det{\left(\Id-\mu_\tau(c)e^{
-s\ell(
c)}\right)}.
\end{align}
The infinite products in \eqref{Ruelle 1} and \eqref{Ruelle 2} converge 
absolutely and locally uniformly only
for $s\in\C$ with $\Real(s)$ sufficiently large. However
one can express each Ruelle zeta function as a
weighted product of Selberg zeta functions. Thus Theorem
\ref{Erstes Theorem} implies the following corollary.
\begin{kor}\label{Theoremzwei}
For every finite dimensional unitary representation $\sigma$ of $M$ and for
every 
finite dimensional irreducible representation $\tau$ of $G$ the 
Ruelle zeta functions $R(s,\sigma)$ and $R(s,\tau)$ admit a meromorphic
continuation to $\C$. 
\end{kor}
We shall now describe an application  of our methods and our main results to 
the  3-dimensional case. In this case there is a natural isomorphism
$G\cong\Sl_2(\C)$. For $m\in\frac{1}{2}\N$ let $\tau(m)$ be 
the $2m$-th symmetric power of the standard representation of $\Sl_2(\C)$. 
Then in \cite{MP} we introduced the analytic torsion $T_X(\tau(m))$ of $X$ 
with coefficients in the locally homogeneous vector bundle defined by 
the restriction of $\tau(m)$ to $\Gamma$. This definition extends 
the definition of the analytic torsion on closed manifolds with coefficients 
in a local system to a specific non-compact situation. Now, recalling
the Cheeger-M\"uller theorems, \cite{Cheeger}, \cite{Muellereins},
\cite{Muellerdrei}, it is natural to ask whether the 
analytic torsion $T_X(\tau(m))$ on the non-compact manifold $X$ has a
combinatorial 
analogue.

In a subsequent paper, we will 
give a partial answer to this question by combining the methods 
and results of this paper with the recent work \cite{MePo} of Menal-Ferrer and
Porti on Reidemeister torsion and Ruelle zeta functions. To 
state our result, 
let $\overline{X}$ be the
Borel-Serre 
compactification of $X$. Then $\overline{X}$ is a compact smooth manifold with
boundary which 
is homotopy-equivalent to $X$ and $X$ is diffeomorphic to the interior of
$\overline{X}$. 
Let $\overline{E}_{\tau(m)}$ denote the flat vector-bundle over $\overline X$
associated 
to the restriction of $\tau(m)$ to $\Gamma$, the fundamental group of
$\overline X$. 
Then Menal-Ferrer and Porti introduced a canoncial family $\{\theta_i\}$ of
bases of the singular
homology groups
of 
$\overline{X}$ with coefficients in $\overline{E}_{\tau(m)}$ \cite[Proposition
2.10]{MePo}. Moreover, they showed that for 
each $m\in\mathbb{N}$, $m\geq 3$, the quotient of the corresponding Reidemeister
torsions $\frac{
\left|\tau_X(\tau(m);\{\theta_i\})\right|}{\left|\tau_X(\tau(2);\{\theta_i\}
)\right|}$ is independent 
of the choice of the family $\{\theta_i\}$ \cite[Proposition 2.2]{MePo}. It 
is also independent of a given spin structure. We 
shall denote this quotient simply by $\frac{
\left|\tau_X(\tau(m)\right|}{\left|\tau_X(\tau(2)
)\right|}$ here. We remark that our parametrization of the representations
$\tau(m)$ is 
different from that used in \cite{MePo} but consistent with \cite{MP1},
\cite{MP}. We will prove the following theorem.
\begin{thrm}\label{Thrmdrei}
Let $X$ be a complete hyperbolic 3-manifold of finite volume and assume that 
its fundamental group $\Gamma$ satisfies \eqref{a1}.
Then for each $m\in\mathbb{N}$, there exists an explicit constant $c(m)$, which 
is independent 
of $X$, such that for $m\geq 3 $ one has
\begin{align*}
\frac{T_X(\tau(m))}{T_X(\tau(2))}=\left(\frac{c(m)}{c(2)}\right)^p\frac{
\left|\tau_X(\tau(m))\right|}{\left|\tau_X(\tau(2))\right|},
\end{align*}
where $p$ denotes the number of cusps of $X$. For
$m\in\frac{1}{2}\mathbb{N}$, 
a similar formula holds. 
\end{thrm}
If one applies the results of \cite{MP} and the explicit form of the constants
$c(m)$, it follows that the quotients of torsions on both side of the equation
are by no means identically one and are 
even exponentially growing as $m\to\infty$. 

The proof of Theorem \ref{Thrmdrei} 
will be given in another publication. It is based on relating the analytic 
torsion $T_X(\tau(m))$ to the behaviour of the Ruelle zeta function 
$R(s,\tau(m))$ at $0$. To prove this relation, we will express the Ruelle 
zeta function $R(s,\tau(m))$ by twisted Selberg zeta functions and then 
we will apply the methods and results of this paper. In particular, we will 
deduce a functional equation and a determinant formula for the Selberg 
zeta functions. 

One can apply Theorem \ref{Thrmdrei} to study for fixed $m\in\mathbb{N}$ the
growth of the
torsion in the 
cohomology $H^*(\Gamma_i,M_{\tau(m)})$ for special
sequences
 of arithmetic groups $\Gamma_i$, $i\in\mathbb{N}$. Here $M_{\tau(m)}$ denotes a
lattice in the 
representation space $V_{\tau(m)}$ of $\tau(m)$ 
which is stable under the  $\Gamma_i$. This gives a modified extension of
some results of Bergeron and Venkatesh \cite{BV} to the case of non-compact
hyperbolic 3-manifolds of finite volume. \\

Let us now recall some of the previous results related to Theorem
\ref{Erstes Theorem}.
If $X$ is compact, it was shown in \cite{Fried2}
that $Z(s,\sigma)$ has a
meromorphic continuation to
$\C$. In a next step, combining the trace formula with spectral-theoretic
methods, for compact X Bunke and Olbrich described the
singularities of $Z(s,\sigma)$ as in the first item of Theorem
\ref{Erstes Theorem} \cite{Bunke}. If  $X$ is only of finite volume and
satisfies assumption \eqref{a1}, 
the meromorphic continuation of the Selberg zeta function and a description 
of its singularities has previously been obtained only for a finite number 
of special representations of $M$. Namely Gangolli
and Wallach \cite{GaWa} solved this problem for the trivial representation of
$M$ and
Gon and Park \cite{Gon} generalized their methods to treat the fundamental
representations $\sigma_{k}$, $k=1,\dots,n$ of $M$ on
$\Lambda^{k}(\mathbb{C}^{2n})$.
However, it is not clear whether the methods of Gangolli, Wallach, Gon and Park
can be
applied to
general $\sigma\in\hat{M}$ since they use a special type of a Paley-Wiener
theorem for differential forms which prescribes the $K$-types of a test function
in a very specific way. This theorem exists only for the fundamental
representation 
$\sigma_k$, $k=1,\dots,n$. However, for an application as our proof of Theorem
\ref{Thrmdrei},
it is 
necessary to understand the Selberg zeta function for every representation
$\sigma$ of 
$M$. \\

To prove our main result, we will
generalize the approach of Bunke and 
Olbrich \cite{Bunke} from the compact case to our situation. In 
the approach of Bunke and Olbrich to the Selberg zeta functions, the generalized
Laplace operators appear quite naturally. To carry over their methods to the
non-compact 
situation we use the invariant trace formula as it is stated and proved by
Hoffmann
\cite{Hoffmann2}. We would like to emphasize that the invariant trace formula 
can be applied only because the Fourier transform of the invariant part
associated to 
the weighted orbital integral has been computed explicitly by Hoffmann
\cite{Hoffmann}. Our approach to the Selberg zeta function is different from
that used
by Gon and Park, since the trace formula 
they use is not invariant and so they have to prescribe the $K$-types 
of their test functions.  \\
 
We shall now describe our proof more precisely. Let us first recall the setting
of the trace formula. The right regular representation
$\pi_\Gamma$ of $G$ on
$L^2(\Gamma\backslash G)$ splits as  
$\pi_{\Gamma}=\pi_{\Gamma,d}\oplus\pi_{\Gamma,c}$. 
Here the representation $\pi_{\Gamma,d}$ is completely reducible. 
On the
other hand, the representation $\pi_{\Gamma,c}$ is isomorphic to a direct
integral
over all tempered principle series representations of $G$. 
For a 
K-finite Schwarz function $\phi$ on $G$, the operator $\pi_{\Gamma,d}(\phi)$ is
trace 
class and the invariant trace formula as it is stated in \cite{Hoffmann2} 
expresses 
$\Tr\left(\pi_{\Gamma,d}(\phi)\right)$
as a sum of invariant distributions on $G$ applied to $\phi$. 

In order to
prove 
the meromorphic continuation of the Selberg zeta function $Z(s,\sigma)$, we
first study
the symmetrized Selberg zeta function $S(s,\sigma)$, which is given by
$Z(s,\sigma)$ if $\sigma=w_0\sigma$ and by $Z(s,\sigma)Z(s,w_0\sigma)$, 
if $\sigma\neq w_0\sigma$. Here $w_0$ is a fixed representative
of the restricted Weyl group. 
As in the compact case \cite{Bunke}, there is a 
$K$-finite function $h_t^\sigma$, belonging to all Harish-Chandra Schwarz
spaces,
such that the logarithmic derivative of $S(s,\sigma)$ is equal to a certain
integral
transform of $H(h_t^\sigma)$. Here $H$ is a distribution on $G$ which occurs in
the
invariant Selberg
trace formula. It is built from the semisimple
conjugacy classes of $\Gamma$. 
Geometrically, the function $h_t^\sigma$ arises from the graded fibre trace of
the kernel of
$e^{-t\tilde{A}(\sigma)}$, where $\tilde{A}(\sigma)$ is a Laplace-type operator
which 
acts on a graded vector bundle $\tilde{E}(\sigma)$ 
over $\widetilde X$. Now we apply the invariant trace formula to $h_t^\sigma$. 
We study the integral transform of all involved
summands carefully and in this way we obtain an expression for the logarithmic
derivative of $S(s,\sigma)$. 

If
$\sigma$ is not
invariant under the Weyl group, we introduce the antisymmetric Selberg zeta
function
$S_a(s,\sigma):=Z(s,\sigma)/Z(s,w_0\sigma)$. The bundle $\tilde{E}(\sigma)$
turns
out to be a
spinor bundle and there is a canonical twisted Dirac operator
$\tilde{D}(\sigma)$ on $\tilde{E}(\sigma)$ such that
$\tilde{D}(\sigma)^2=\tilde{A}(\sigma)$.
Now the fibre trace of the kernel of
$\Tr(\tilde{D}(\sigma)e^{-t\tilde{D}(\sigma)^2})$ is 
represented by a $K$-finite Harish-Chandra Schwarz function $k_t^\sigma$ and the
logarithmic derivative of 
$S_a(s,\sigma)$ equals an integral
transform of $H(k_t^\sigma)$, where the distribution $H$ is as above. Using the
invariant 
trace formula again we can study the logarithmic derivative of $S_a(s,\sigma)$.
Putting everything together,
we can complete the proof of Theorem \ref{Erstes Theorem}.
\\

This paper is organized as follows. In section \ref{secpr} we fix some notation
and recall 
some basic facts concerning the group $G$.  In section \ref{secselz} we
describe 
the Selberg zeta functions in terms which are needed for our approach via the
trace formula. In section \ref{secC} we review Hoffman's factorization 
of the $C$-matrix associated to the Eisenstein series which is an important 
tool to state the invariant trace formula. The latter formula is treated in
section \ref{secinv}. 
In the same section we study the Fourier transform of a certain
distribution
$\mathcal{I}$ 
occuring in the trace formula. In section \ref{secBLO} we introduce 
certain Laplace operators which act on graded locally homogeneous vector bundles
over 
$X$ and compute the Fourier transform of the corresponding heat kernels.
The meromorphic 
continuation of the symmetric Selberg zeta function and a description of 
its singularities is treated in section \ref{secsymzeta}. Section 
\ref{sectwDO} is devoted to the study of certain twisted Dirac operators on
locally homogeneous bundles over $X$. In the final 
section \ref{secasyms} we study the antisymmetric Selberg zeta function and
complete the proof of Theorem \ref{Erstes Theorem}.

\bigskip
{\bf Acknowledgement.}
This paper contains parts of the author's PhD thesis. He would like to thank his 
supervisor Prof. Werner M\"uller for his constant support and for helpful
suggestions.

\section{Preliminaries}\label{secpr}
\setcounter{equation}{0}
In this section we establish some notation and recall some basic facts 
about representations of the involved Lie groups.
\addtocontents{toc}{\protect\setcounter{tocdepth}{1}}
\subsection{}
For $d\in\N$, $d=2n+1$ we let $G:=\Spin(d,1)$. 
The group $G$ is defined as the universal covering group 
of $\SO_{0}(d,1)$, where $\SO_{0}(d,1)$ is the identity component of $\SO(d,1)$.
Let
$K:=\Spin(d)$. Then $K$ is a maximal compact subgroup of $G$. Put
$\widetilde X:=G/K$. Let $G=NAK$
be the standard Iwasawa decomposition of $G$ and let $M$ be the 
centralizer of $A$ in $K$. Then we have $M=\Spin(d-1)$. 
The Lie algebras of  $G,K,A,M$ and $N$ will be denoted by
$\gL,\kL,\aL,\mL$ and $\nf$, respectively. Define the 
standard Cartan involution $\theta\colon\gL\rightarrow \gL$ by
$\theta(Y):=-Y^{t},\quad Y\in\gL$.
Let 
$\mathfrak{g}:=\mathfrak{k}\oplus\mathfrak{p}$
be the Cartan decomposition of $\gL$ with respect to $\theta$. Let $B$ be 
the Killing form of $\gL$. Then $B$ 
is positive on $\pL$. We define a symmetric bilinear form
$\langle\cdot,\cdot\rangle$ on $\gL$ by $\langle
Y_1,Y_2\rangle:=\frac{1}{2(d-1)}B(Y_1,Y_2)$, $Y_1,Y_2\in\gL$.
Let $x_0=eK\in
\widetilde X$. Then we have a canonical isomorphism
$T_{x_0}\widetilde X\cong \pL$
and thus the restriction of $\langle\cdot,\cdot\rangle$ to $\pL$ 
defines an inner product on $T_{x_0}\widetilde X$ and therefore an invariant 
metric on $\widetilde X$. This metric has constant curvature $-1$ and
$\widetilde X$,
equipped with this metric, is isometric to the hyperbolic space
$\mathbb{H}^{d}$. 

\subsection{}\label{subsH}
Let $\aL$ be the Lie-Algebra of $A$. Then $\dim(\aL)=1$.
Fix a Cartan-subalgebra $\mathfrak{b}$ of $\mathfrak{m}$. 
Then $\mathfrak{b}$ is also a Cartan subalgebra of $\kL$ and
$\mathfrak{h}:=\mathfrak{a}\oplus\mathfrak{b}$
is a Cartan-subalgebra of $\mathfrak{g}$. We can identify
$\mathfrak{g}_\C\cong\mathfrak{so}(d+1,\C)$. Let $e_1\in\aL^*$ be the
positive restricted root defining $\mathfrak{n}$ and let 
$H_1\in\aL$ such that $e_1(H_1)=1$.
We fix $e_2,\dots,e_{n+1}\in
i\mathfrak{b}^*$ such that 
the positive roots $\Delta^+(\mathfrak{g}_\C,\mathfrak{h}_\C)$ are chosen as in
\cite[page 102]{Goodman}
for the root system $D$. We let
$\Delta^+(\mathfrak{g}_\C,\mathfrak{a}_\C)$ be
the set of roots of $\Delta^+(\mathfrak{g}_\C,\mathfrak{h}_\C)$ which do not
vanish on $\aL_\C$. The positive roots
$\Delta^+(\mathfrak{m}_\C,\mathfrak{b}_\C)$
are chosen such that they are restrictions of elements from
$\Delta^+(\mathfrak{g}_\C,\mathfrak{h}_\C)$.
For $\alpha\in\Delta^{+}(\mathfrak{g}_{\C},\mathfrak{h}_{\C})$ 
there exists a unique $H'_{\alpha}\in\mathfrak{h}_{\C}$ 
such that $B(H,H_{\alpha}^{'})=\alpha(H)$ for all $H\in\mathfrak{h}_{\C}$. 
One has $\alpha(H_{\alpha}^{'})\neq 0$. We let
$H_{\alpha}:=\frac{2}{\alpha(H_{\alpha}^{'})}H_{\alpha}^{'}$.
For $j=1,\dots,n+1$ we let $\rho_{j}:=n+1-j$.
Then the half-sum of positive roots $\rho_G$ and $\rho_M$ of
$\Delta^{+}(\mathfrak{g}_{\mathbb{C}},
\mathfrak{h}_\mathbb{C})$ resp. $\Delta^+(\mathfrak{m}_{\mathbb{C}},
\mathfrak{b}_{\mathbb{C}})$ 
are given by $\rho_{G}=\sum_{j=1}^{n+1}\rho_{j}e_{j}$ resp
$\rho_{M}=\sum_{j=2}^{n+1}\rho_{j}e_{j}$.

\subsection{}\label{subsrepmk}

Let ${{\mathbb{Z}}\left[\frac{1}{2}\right]}^{j}$ be the set of all 
$(k_{1},\dots,k_{j})\in\mathbb{Q}^{j}$ such that  all $k_{i}$ are 
integers or all $k_{i}$ are half integers. Then the finite dimensional
representations $\nu\in\hat{K}$ of $K$ are
parametrized by their highest weights $\Lambda\nu\in i\bL^*$, 
$\Lambda(\nu)=k_{2}(\nu)e_{2}+\dots+k_{n+1}(\nu)e_{n+1}$,
$(k_2(\nu),\dots,k_{n+1}(\nu))\in{{\mathbb{Z}}\left[\frac{1}{2}\right]}^{n}$, 
$k_2(\nu)\geq k_3(\nu)\geq\ldots\geq k_{n+1}(\nu)\geq 0$.
The  finite dimensional irreducible representations 
$\sigma\in\hat{M}$ of $M$ 
are parametrized by their highest weights $\Lambda(\sigma)\in i\bL^*$, 
\begin{equation}\label{RepM}
\Lambda(\sigma)=k_{2}(\sigma)e_{2}+\dots+k_{n+1}(\sigma)e_{n+1};\:\:
k_{2}(\sigma)\geq 
k_{3}(\sigma)\geq\dots\geq k_{n}(\sigma)\geq \left|k_{n+1}(\sigma)\right|,
\end{equation}
where $(k_{2}(\sigma),\dots,k_{n+1}(\sigma))
\in{{\mathbb{Z}}\left[\frac{1}{2}\right]}^{n}$.
Let $M'$ be the normalizer of $A$ in $K$ and let $W(A)=M'/M$ be the 
restricted Weyl-group. It has order two. Let $w_{0}\in W(A)$ be the non-trivial 
element. Then for $\sigma\in\hat{M}$ there is an associated representation
$w_0\sigma$, see \cite[section 2.4]{MP}. If the highest weight of $\sigma$ is as
in
\eqref{RepM}, then the highest weight $\Lambda(w_0\sigma)$ of $w_0\sigma$ is
given 
by
$\Lambda(w_0\sigma)=k_2(\sigma)e_2+\dots+k_n(\sigma)e_n-k_{n+1}(\sigma)e_{n+1}$.
For $\nu\in\hat{K}$ and $\sigma\in\hat{M}$ we denote 
by $\left[\nu:\sigma\right]$ the multiplicity of $\sigma$ in the restriction 
of $\nu$ to $M$.

\subsection{}\label{Subsecbr}

Let $\kappa$ be the spin-representation of $K$ over the spinor space
$\Delta^{2n}$. Then $\kappa$ is the
representation with highest weight
$\Lambda(\kappa)=\frac{1}{2}e_{2}+\dots+\frac{1}{2}e_{n+1}.$ 
By \cite[Theorem 8.1.4]{Goodman}
there is
an
$M$-invariant splitting 
$\Delta^{2n}=\Delta^{2n}_{+}\oplus\Delta^{2n}_{-}$
such that the restriction of $\kappa$ to $M$ acts on $\Delta^{2n}_{+}$ as
$\kappa^{+}$ and on $\Delta^{2n}_{-}$ as $\kappa^{-}$, where $\kappa^{\pm}$ are
the
representation of $M$ with highest weights
$\frac{1}{2}e_{2}+\dots+\frac{1}{2}e_{n}\pm\frac{1}{2}e_{n+1}$. 
Let $R(K)$ and $R(M)$ be the representation rings of $K$ and $M$. Let 
$\iota\colon M\longrightarrow K$ be the inclusion and let $\iota^{*}\colon R(K)
\longrightarrow R(M)$ be the induced map. If $R(M)^{W(A)}$ is the subring of 
$W(A)$-invariant elements of $R(M)$, then clearly $\iota^{*}$ maps $R(K)$ into 
$R(M)^{W(A)}$. Moreover, the following proposition holds.
\begin{prop}\label{branching}
The map $\iota$ is an isomorphism from $R(K)$ onto $R(M)^{W(A)}$.
Let $\sigma\in\hat{M}$ be of highest weight $\Lambda(\sigma)$ as in
\eqref{RepM} and assume that $k_{n+1}(\sigma)> 0$. Let
$\nu(\sigma)\in\hat{K}$
be the representation of highest weight $
\Lambda\left(\nu(\sigma)\right):=\sum_{j=2}^{n+1}
\left(k_j(\sigma)-1/2\right)e_j$. Then one has
$\sigma-w_0\sigma=\left(\kappa^+-\kappa^-\right)\otimes\iota^*\nu(\sigma)$.
Moreover, $\nu(\sigma)\otimes\kappa$ splits as
$\nu(\sigma)\otimes \kappa=\nu^+(\sigma)\oplus\nu^-(\sigma)$
such that $\sigma+w_0\sigma=\iota^*\nu^+(\sigma)-\iota^*\nu^-(\sigma)$.
\end{prop}
\begin{proof}
This is proved by Bunke and Olbrich, \cite{Bunke} , Proposition 1.1.
\end{proof}

\subsection{}\label{subsecprs}

We parametrize the principal series as follows. Given $\sigma\in\hat{M}$ with
$(\sigma,V_\sigma) \in \sigma$, let $\mathcal{H}^{\sigma}$ denote the space of
$V_\sigma$-valued $L^2$-functions on $K$ which satisfy
$\Phi(mk)=\sigma(m)\Phi(k)$ for allmost all  $k\in K$, all $m\in M$.
Then for $\lambda\in\mathbb{C}$ and $\Phi\in H^{\sigma}$ we define
$\pi_{\sigma,\lambda}(g)\Phi(k):=e^{(i\lambda+n)H(kg)}\Phi(\kappa(kg))$.
Recall that the representations $\pi_{\sigma,\lambda}$ are unitary iff 
$\lambda\in\mathbb{R}$. Moreover, by \cite[Theorem 7.2, Theorem 7.12]{Knapp},
$\pi_{\sigma,\lambda}$ and $\pi_{\sigma',\lambda'}$, $\lambda,
\lambda'\in\mathbb{C}$ are 
equivalent iff either $\sigma=\sigma'$, $\lambda=\lambda'$ or
$\sigma'=w_{0}\sigma$, $\lambda'=-\lambda$. Let $\nu\in\hat{K}$. The
restriction of
$\pi_{\sigma,\lambda}$ to $K$ coincides with the
induced 
representation ${\rm{Ind}}_{M}^{K}(\sigma)$ and thus by Frobenius 
reciprocity \cite[p.208]{Knapp} the multiplicity of 
$\nu$ in $\pi_{\sigma,\lambda}$  equals $\left[\nu:\sigma\right]$. By
\cite[Theorem 8.1.4]{Goodman} one has $\left[\nu:\sigma\right]\leq 1$. Let
$c(\sigma):=\sum_{j=2}^{n+1}(k_{j}(\sigma)+\rho_{j})^{2}-\sum_{j=1}^{n+1}
\rho_{j}^{2}$. Then, if $\Omega$ is the Casimir element with respect to the
normalized Killing form
one has $\pi_{\sigma,\lambda}(\Omega)=-\lambda^2+c(\sigma)$,
see \cite[Corollary 2.4]{MP1}. By $\Theta_{\sigma,\lambda}$ we will denote the
global character of $\pi_{\sigma,\lambda}$. 

\subsection{}
We let $\Gamma$ be a discrete, torsion free subgroup of $G$ with
$\vol(\Gamma\backslash G)<\infty$ and we assume that $\Gamma$ satisfies
\eqref{a1}
We let 
\begin{align*}
X:=\Gamma\backslash G/K.
\end{align*} 
We equip
$X$ with the Riemannian metric
induced from $\widetilde{X}$. Let $\mathfrak{P}$ be a fixed
set of representatives of $\Gamma$-inequivalent cuspidal parabolic subgroups of
$G$. Then $\mathfrak{P}$ is finite. Let 
$p:=\#\mathfrak{P}$. 
Then $p$ equals the number of cusps of $X$. Let 
$P_0:=MAN$. Without loss of
generality we will assume
that $P_{0}\in\mathfrak{P}$. 
For every $P\in\mathfrak{P}$, there exists a $k_{P}\in K$ such that
$P=N_{P}A_{P}M_{P}$ with $N_{P}=k_{P}Nk_{P}^{-1}$,
$A_{P}=k_{P}Ak_{P}^{-1}$, $M_{P}=k_{P}Mk_{P}^{-1}$. We let
$k_{P_{0}}=1$.
If $P\in\mathfrak{P}$, $P'\in\mathfrak{P}$ we will say that
$\sigma_{P}\in\hat{M}_{P}$ and $\sigma_{P'}\in\hat{M}_{P'}$ are associated if
for all $m_{P'}\in M_{P'}$ one has
$\sigma_{P'}(m_{P'})=\sigma_{P}(k_{P}k_{P'}^{-1}m_{P'}k_{P'}k_{P}^{-1})$. For
$\sigma_{P}\in\hat{M}_{P}$ we will denote by $\boldsymbol{\sigma}$ the set of
all $\sigma_{P'}\in\hat{M}_{P'}$ associated to $\sigma_{P}$, where $P'$ runs
through $\mathfrak{P}$. For $g\in G$, we define
$n_{P}(g)\in N_{P}$, $H_{P}(g)\in
\mathbb{R}$ and $\kappa_{P}(g)\in K$ by
$g=n_{P}(g)\exp{(H_{P}(g)H_1)}\kappa_{P}(g)$. 
We let $H(g):=H_{P_0}(g)$.

\section{Selberg zeta functions}\label{secselz}
\setcounter{equation}{0}
This section is devoted to a preliminary description and investigation of the
functions 
$Z(s,\sigma)$. We proceed analogously to \cite[section (3.1)]{Bunke}.\\ 
Let $\Phi$ be the geodesic flow on $SX$. Since $K$ acts transitively on the
unit-sphere of $\pL$, there is a
canonical isomorphism $SX\cong\Gamma\backslash G/M$
and 
in this way one obtains an isomorphism
\begin{align}\label{An}
TSX\cong \Gamma\backslash
G\times_{\Ad}(\overline{\mathfrak{n}}\oplus\mathfrak{n}\oplus\mathfrak{a}).
\end{align}
We fix an $M$-invariant inner product on
$\overline{\mathfrak{n}}\oplus\mathfrak{n}\oplus\mathfrak{a}$ and 
equip $TSX$ with the induced metric. With respect to \eqref{An},
$\Phi$ is given on
$\Gamma\backslash G/M$ as
$\Phi(t,\Gamma gM)=\Gamma g\exp{-tH_1}M$,
where $\Phi$ is well-defined since $M$ and $A$ commute. Thus, with respect to
\eqref{An}, $d\Phi$, regarded 
as a flow on $TSX$, is given by
\begin{align*}
d\Phi(t,\left[\Gamma g,Y\right])=\left[\Gamma
g\exp{-tH_1},\Ad(\exp{tH_1})Y\right],\quad Y\in
\overline{\mathfrak{n}}\oplus\mathfrak{a}\oplus\mathfrak{n},\:g\in G.
\end{align*}
The spaces $\overline{\mathfrak{n}}$, $\mathfrak{n}$ and $\mathfrak{a}$ are
invariant
under $\Ad(A)$ and $\Ad(\exp{tH_1})$ acts on these spaces by $e^{-t}\cdot\Id$
respectively $e^{t}\cdot\Id$ respectively $\Id$. Thus the decomposition 
on the right hand side of \eqref{An} gives the Anosov-decomposition in
\eqref{splitting}.\\
Let $\sigma$ be a unitary finite dimensional
representation of $M$
on $V_{\sigma}$
and let $V(\sigma):=\Gamma\backslash(G\times_{\sigma}V_\sigma)$. Then
$V(\sigma)$ 
is a vector bundle over $SX$ and $\Phi$ lifts to a flow $\Phi_\sigma$ on
$V(\sigma)$ which is defined by
$\Phi_{\sigma}(t,\left[\Gamma g,v\right]):=\left[\Gamma
g\exp(-tH_1),v\right]$.\\
Next we describe the closed geodesics of $X$ in terms of the conjugacy classes
of $\Gamma$ 
which we will denote by $\CC(\Gamma)$. There is
a canonical
one-to one correspondence between $\CC(\Gamma)$ and the set of free homotopy
classes of
closed paths in $X$. 
For $\gamma\in\Gamma$ we will denote its conjugacy class by
$\left[\gamma\right]$. Moreover, by $f(\left[\gamma\right])$ we will denote the 
free homotopy class of closed paths associated to $\left[\gamma\right]$.
Now for $\left[\gamma\right]\in \CC(\Gamma)$ we let
$\ell(\gamma)$ be the infimum over all lengths
of the piecewise smooth curves belonging to $f(\left[\gamma\right])$.
We let $\CC(\Gamma)_{\s}$ be the set of conjugacy classes $\left[\gamma\right]$
such 
that $\gamma$ is semisimple. Moreover we let $\CC(\Gamma)_{\parab}$ be the set
of conjugacy classes 
$\left[\gamma\right]$ such that $\gamma$ is $\Gamma$-conjugate 
to an element of $\Gamma\cap N_P$, $P\in\mathfrak{P}$. Then by \cite[Lemma
5.3]{Warner1}
and our assumption \eqref{a1} 
we have $\CC(\Gamma)=\CC(\Gamma)_{\s}\cup \CC(\Gamma)_{\parab}$ and $
\CC(\Gamma)_{\s}\cap\CC(\Gamma)_{\parab}=\left[1\right]$. For 
$\left[\gamma\right]\in\CC(\Gamma)_{\parab}$ it is easy to see that
$\ell(\gamma)=0$.\\ On the other hand let $\gamma\in\Gamma$
be 
semimisple, $\gamma\neq 1$. Then by the same argument as in the proof of 
\cite[Lemma 6.6]{Wallach} for the cocompact case, it follows that there exists
$g\in G$, $t_\gamma\in (0,\infty)$ and $m_\gamma \in M$ 
such that $g\gamma g^{-1}=m_\gamma \exp{t_\gamma H_1}$, where $t_\gamma$ is
unique and $m_\gamma$ is 
unique up to conjugation in $M$. Thus the geodesic
$\widetilde{c}_{\gamma}(t):=g^{-1}\exp{(tH_1)}K$  
in $\widetilde X$ is stabilized by $\gamma$ and projects to a closed geodesic
$c_{\gamma}\in f(\left[\gamma\right])$ 
of length $t_\gamma$. Applying \cite[Proposition 4.2]{Bishop}
we conclude that the elements in $\widetilde{X}$ which minimize $d(x, \gamma x)$
are exactly the
elements $\widetilde{c}_{\gamma}(t)$, $t\in\R$. It follows that
$\ell(\gamma)=t_\gamma$. Proceeding as in the proof of \cite[Lemma 4.1]{Gang},
one can show that for every semisimple element $\gamma\in\Gamma$ its
centralizer 
$Z(\gamma)$ in $\Gamma$ is infinite cyclic and generated by a semisimple
element 
$\gamma_0$. Thus we may define $n_\Gamma\in\N$ by
$\gamma=\gamma_0^{n_\Gamma(\gamma)}$. 
We call $\gamma$ resp. $\left[\gamma\right]$ prime if $n_\Gamma(\gamma)=1$. Then
it is easy to see that 
this is equivalent to saying that $c_{\gamma}$ is a prime
geodesic.\\
Putting everything together, it follows that the Selberg zeta function from
equation \eqref{SZF}
can be written as
\begin{align}\label{SZF2}
Z(s,\sigma)=\prod_{\substack{\left[\gamma\right]\in\CC(\Gamma)_{\s}-\left[
1\right]\\
\left[\gamma\right]\:\text{prime}}}\prod_{k=0}^\infty\det{
\left(\Id-\sigma(m_\gamma)\otimes
S^k\Ad(m_\gamma\exp(\ell(\gamma)H_1))|_{\bar{\mathfrak{n}}}e^{
-(s+n)\ell(\gamma)}\right)}.
\end{align}
In order to establish the convergence of the infinite product in \eqref{SZF2}, 
we remark that
\begin{align}\label{estdet}
{{\det(\Id-\Ad(m_{\gamma}a_{\gamma})|_{{\bar{\mathfrak
{n}}}})}}^{-1}\leq \left(1-e^{-l(\gamma)}\right)^{-n}.
\end{align}
Moreover, using the volume growth 
$\vol B_R(x)\leq c e^{2nR}$ on $\widetilde X$, the argument of \cite[Lemma 4.3,
Proposition 4.4]{MePo} carries over to 
higher dimensions and one obtains an estimate 
\begin{align}\label{LS}
\#\{\left[\gamma\right]\in\CC(\Gamma)_{\s}\colon \ell(\gamma)\leq
R\}\leq C e^{2nR}
\end{align}
for all $R$ and some constant $C>0$. 
Now arguing as in \cite[equation (3.6)]{Bunke} one computes
\begin{align}\label{LogSZF}
\log{Z(s,\sigma)}=-\sum_{\left[\gamma\right]\in\CC(\Gamma)_{\s}-\left[
1\right]}\frac{\Tr\left(\sigma(m_{\gamma})\right)e^{
-(s+n)\ell(\gamma)}}{n_{\Gamma}(\gamma){\det(\Id-\Ad(m_{\gamma}a_
{\gamma})|_{{\bar{\mathfrak{n}}}})}}
\end{align}
and applying \eqref{estdet} and \eqref{LS} it follows that the infinte product
in \eqref{SZF} resp. \eqref{SZF2} converges absolutly and 
locally uniformly on the set $\Real(s)>2n$. \\
The formula \cite[equation
(3.4)]{Bunke} 
for the logarithmic derivative of $R(s,\sigma)$ remains valid 
also in our case. A similar formula can also be obtained for the logarithmic
derivative of $R(s,\tau)$. Thus it follows from \eqref{LS} that the infinit
products in \eqref{Ruelle 1} converges absolutely for $\Real(s)>2n$ and that the
infinite product in \eqref{Ruelle 2} converges absolutely for $s\in\C$
with 
$\Real(s)$ sufficiently large. 
Also, arguing as in cocompact case, \cite[Proposition 3.4]{Bunke}, \cite[section
6]{Wotzke}, we 
can express the functions $R(s,\sigma)$ and $R(s,\tau)$ as weighted products of
Selberg zeta functions.

\section{Some properties of the C-matrix}\label{secC}
\setcounter{equation}{0}
In this section we describe some of the main properties of the $C$-matrix
associated to 
the Eisenstein series which are needed for our application of the invariant
trace formula.\\ 
For $P\in\mathfrak{P}$, $\nu\in\hat{K}$, $\sigma_{P}\in\hat{M}_P$,  with
$\left[\nu:\sigma_P\right]\neq 0$ 
we let $\mathcal{E}_{P}(\nu,\sigma_{P})$ be the set of all
continuous functions $\Phi$ on $G$ which are left-invariant under $N_PA_P$ such
that for all
$x\in G$ the function
$m\mapsto \Phi_{P}(mx)$ belongs to $L^2(M,\sigma_P)$, the
$\sigma_{P}$-isotypical component of the
right regular representation of $M_{P}$, and such that 
for all $x\in G$ the function $k\mapsto \Phi_{P}(xk)$ belongs to the
$\nu$-isotypical component of the right regular representation of $K$. 
We define an inner product on $\mathcal{E}_{P}(\nu,\sigma_{P})$ as follows.
Every element of $\mathcal{E}_{P}(\nu,\sigma_{P})$ can be identified canonically
with
a function on $K$. For $\Phi,\Psi\in\mathcal{E}_{P}(\nu,\sigma_{P})$ we now set
$\left<\Phi,\Psi\right>:=\vol(\Gamma\cap N_P\backslash
N_P)\int_K\Phi(k)\bar{\Psi}(k)dk$.
Now we define 
a Hilbert space $\mathcal{E}_P(\sigma_P)$ by
\begin{align*}
\mathcal{E}_P(\sigma_P):=\bigoplus_{\substack{\nu\in\hat{K}\\\left[
\nu:\sigma_P\right]\neq 0}}\mathcal{E}_P(\nu,\sigma_P).
\end{align*}
For $\lambda\in\C$ let $\pi_{\Gamma,\sigma_{P},\lambda}$ be the representation
of $G$ on
$\mathcal{E}_{P}(\sigma_{P})$ defined by
\begin{align*}
\pi_{\Gamma,\sigma_{P},\lambda}(g)\Phi(n_{P}a_{P}k):=e^{(\lambda+n)H_{P}(kg)}
\Phi(kg),\quad n_{P}\in N_{P},\: a_{P}\in A_{P}, k\in K,\quad
\Phi\in\mathcal{E}_{P}(\sigma_P,\nu).
\end{align*}
For $\sigma\in\hat{M}$ and $\nu\in\hat{K}$ with $\left[\nu:\sigma\right]\neq 0$
put
\begin{align*}
\boldsymbol{\mathcal{E}}(\nu:\sigma
):=\bigoplus_{\sigma_P\in\boldsymbol{\sigma}}\mathcal{E}(\nu:\sigma_P);\quad
\boldsymbol{\mathcal{E}}(\sigma):=\bigoplus_{
\sigma_P
\in\boldsymbol{\sigma}}\mathcal{E}_{P}(\sigma_{P}); \quad
\boldsymbol{\pi}_{\Gamma,\sigma,\lambda}:=\bigoplus_{\sigma_{P}
\in\boldsymbol{\sigma}}\pi_{\Gamma,\sigma_{P},\lambda}.
\end{align*}
Then
$\boldsymbol{\pi}_{\Gamma,
\sigma,\lambda}$ is a representation of $G$ on
$\boldsymbol{\mathcal{E}}(\sigma)$.
The pair
$(\boldsymbol{\mathcal{E}}(\sigma),\boldsymbol{\pi}_{\Gamma,\sigma,\lambda})$
can be related to the the principal series as
follows. For $\sigma\in\hat{M}$ let $\Hom_{M}(V_\sigma,L^{2}(M))$ denote the set
of intertwining operators between $\sigma$ and the right regular
representation. The dimension of this space equals the degree of
$\sigma$. For $\nu\in\hat{K}$ with $\left[\nu:\sigma\right]\neq 0$ let
$(\mathcal{H}^\sigma)^\nu$ denote the $\nu$-isotypical component of
$\mathcal{H}^\sigma$, where 
the latter space is as in section \ref{subsecprs}. Then we define
$I_{P,\sigma}(\nu)\colon\Hom_{M}(V_\sigma,L^{2}(M))\otimes
(\mathcal{H}^{\sigma})^\nu\rightarrow \mathcal{E}_{P}(\sigma_{P})$
by
\begin{align}\label{Erste Definition des Isomorphismus I(P)}
I_{P,\sigma}(u\otimes
\Phi)(g):=u\circ\Phi(mk_{p}^{-1}\kappa_{P}(g))(m^{-1})\:\text{for almost all
$m\in M$}.
\end{align}
Then $I_{P,\sigma}(\nu)$ is an isometry. We let $I_{P,\sigma}$ be the direct sum
of 
the $I_{P,\sigma}(\nu)$ and define
\begin{align*}
\boldsymbol{\mathcal{L}}(\sigma):=\bigoplus_{P\in\mathfrak{P}}\Hom_{
M}(V_\sigma,L^{2}(M));\quad
\mathbf{I}_{\sigma}:\boldsymbol{\mathcal{L}}(\sigma)\otimes
\mathcal{H}^{\sigma}\longrightarrow
\boldsymbol{\mathcal{E}}(\sigma),\:\mathbf{I}_{\sigma}=\bigoplus_{
P\in\mathfrak{P}}I_{P,\sigma}
\end{align*}
The map $\mathbf{I}_{\sigma}$ is an isomorphism and an intertwining operator
between the
representations $1\otimes \pi_{\sigma,\lambda}$ and
$\boldsymbol{\pi}_{\Gamma,\sigma,i\lambda}$ , where $1$ stands for the
trivial representation of $G$ on $\boldsymbol{\mathcal{L}}(\sigma)$.\\
Out of the constant terms associated to the Eisenstein series, one can construct
operators
\begin{align*}
\mathbf{C}(\nu:\sigma:\lambda):\boldsymbol{\mathcal{E}}(\nu:\sigma
) \to\boldsymbol{\mathcal{E}}(\nu:w_0\sigma
);\quad
\mathbf{C}(\sigma:\lambda):=\bigoplus_{\substack{\nu\in\hat{K}\\ \left[
\nu:\sigma\right]\neq 0}}\mathbf{C}
(\nu:\sigma:\lambda)\colon\boldsymbol{\mathcal{E}}(\sigma
) \to\boldsymbol{\mathcal{E}}(w_0\sigma
),
\end{align*}
see \cite[section 3]{MP}. For every $\sigma$ the function
$\lambda\mapsto
\mathbf{C}(\sigma:\lambda)$ is meromophic in $\lambda\in\C$ and 
has no poles on $i\R$. Moreover it satisfies the functional equation
\begin{align}\label{FEEs}
\mathbf{C}(w_0\sigma:\lambda)\mathbf{C}(\sigma
:-\lambda)=\Id;\quad\mathbf{C}(\sigma:\lambda)^{*}=\mathbf{C}(w_0\sigma:\bar{
\lambda}).
\end{align}
The representations
$\boldsymbol{\pi}_{\Gamma,\sigma,\lambda}$ and
$\boldsymbol{\pi}_{\Gamma,w_0\sigma,-\lambda}$ are equivalent
and $\mathbf{C}(\sigma:\lambda)$ is an intertwining operator
between $\boldsymbol{\pi}_{\Gamma,\sigma,\lambda}$ and
$\boldsymbol{\pi}_{\Gamma,w_0\sigma,-\lambda}$.\\
Let $\sigma\in\hat{M}$. According to \cite{Hoffmann2}, the map
$\mathbf{I}_{\sigma}$ can be used to relate the
intertwining operators $\mathbf{C}(\sigma:s)$ to the Knapp-Stein
intertwining operators of the principal series representations associated to
$P_0$. One obtains a result similar to the adelic case, where the
C-matrix
factorizes into a product of the Knapp-Stein intertwining operator and the
intertwining operators at the finite places. We first briefly recall the
definition of the Knapp-Stein operators. Let $\Theta$ denote the lift of
$\theta$ to $G$, let
$\bar{N}:=\Theta(N)$ and let
$\bar{P}_0:=\bar{N}AK$ be the parabolic subgroup opposite to
$P_0$. Let $\sigma\in\hat{M}$. For $\Phi\in\mathcal{H}^\sigma$ we define a 
function $\Phi_{\lambda}$ on $G$ by
$\Phi_\lambda(nak):=e^{\left(i\lambda+n\right)H(a)}\Phi(k)$. Let
$(\mathcal{H}^\sigma)^K$ denote 
the $K$-finite vectors in $\mathcal{H}^\sigma$. Let $m_0\in K$ be a
representative for $w_0$. Then for $\Iim(\lambda)<0$ and
$\Phi\in(\mathcal{H}^\sigma)^K$ the
integral
\begin{align}\label{IntIO}
J_{\bar{P}_{0}|P_{0}}(\sigma,\lambda)(\Phi)(k):=\int_{\bar{N}}{\Phi_{\lambda}
(\bar{n}k)d\bar{n}}=\int_{N}{\Phi_{\lambda}(m_{0}
nm_{0}^{-1}k)dn}
\end{align}
is convergent and $J_{\bar{P}_{0}|P_{0}}(\sigma,\lambda)$ extends to an
intertwining operator
$J_{\bar{P}_{0}|P_{0}}(\sigma,\lambda):\mathcal{H}^{\sigma}
\longrightarrow \mathcal{H}^{\sigma}$ between $\pi_{\sigma,\lambda}$ and
$\pi_{\sigma,\lambda,\bar{P}_{0}}$,
where $\pi_{\sigma,\lambda,\bar{P}_{0}}$ denotes the principal series
representation associated to $\sigma$, $\lambda$ and $\bar{P}_{0}$.
Moreover, by \cite{Knapp Stein}, as an operator-valued function
$J_{\bar{P}_{0}|P_{0}}(\sigma,\lambda)$ has a meromorphic continuation to
$\mathbb{C}$. If $\sigma\neq w_{0}\sigma$,
$J_{\bar{P}_{0}|P_{0}}(\sigma,\lambda)$ has no poles on $i\mathbb{R}$ and is
invertible there. If $\sigma=w_{0}\sigma$,
$J_{\bar{P}_{0}|P_{0}}(\sigma,\lambda)$ is regular and invertible on
$i\mathbb{R}-\{0\}$. Next one defines an operator
$A(w_{0}):\mathcal{H}^{\sigma}\rightarrow \mathcal{H}^{w_{0}\sigma}$ by
$A(w_{0})\Phi(k):=\Phi(m_{0}k)$. Then $A(w_{0})$ intertwines
$\pi_{\sigma,\lambda,\bar{P}_0}$
and $\pi_{w_0\sigma,-\lambda}$. Thus the operator 
\begin{align}\label{DefJ}
J_{P_{0}}(\sigma,\lambda)\colon
\mathcal{H}^\sigma\to\mathcal{H}^{w_0\sigma},\quad
\:J_{P_{0}}(\sigma,\lambda):=A(w_{0
})J_{\bar{P}_{0}|P_{0}}(\sigma,\lambda)
\end{align}
is, wherever it is defined, an intertwining operator between
$\pi_{\sigma,\lambda}$ and $\pi_{w_{0}\sigma,-\lambda}$. 
If $\nu\in\hat{K}$ with $\left[\nu:\sigma\right]\neq 0$ we denote by
$J_{P_{0}}(\nu,\sigma,\lambda)$ the
restriction of
$J_{P_{0}}(\sigma,\lambda)$
to a map from
$(\mathcal{H}^{\sigma})^{\nu}$
to $(\mathcal{H}^{w_{0}\sigma})^{\nu}$. 
Now the C-matrix is related to the Knapp-Stein operator as follows.
\begin{prop}\label{Faktorisierung C-Term}
There exist a meromorphic
$\Hom(\boldsymbol{\mathcal{L}}(\sigma),\boldsymbol{\mathcal{L}}
(w_0\sigma))$-valued function $\mathbf{T}(\sigma,\lambda)$ which is regular
on $i\R-\{0\}$ such that in 
the sense of meromorphic functions one has
\begin{align*}
\mathbf{C}(\nu:\sigma:\lambda)(\mathbf{I}_{\sigma}(u\otimes\Phi))=\mathbf{I}_{
w_0
\sigma}\left(\mathbf
{T}(\sigma,\lambda)(u)\otimes J_{P_{0}}(\nu,\sigma,-i\lambda)(\Phi)\right)
\end{align*}
for all $u\in\boldsymbol{\mathcal{L}}(\sigma)$ and all
$\Phi\in(\mathcal{H}^{\sigma})^{\nu}$.
\end{prop}
\begin{proof}
The proposition is proved in \cite[Theorem 7.1]{Hoffmann2} for the more general
setting of a rank-one
lattice and a Hecke operator. 
\end{proof}
Proposition \ref{Faktorisierung C-Term} gives the following corollary.
\begin{kor}\label{Korollar fur Streuterm}
Let $\alpha$ be a $K$-finite Schwarz-function. Then in the sense of meromorphic
functions one has 
\begin{align*}
&\Tr\left(\boldsymbol{\pi}_{\Gamma,\sigma,\lambda}(\alpha)\mathbf{C
}(
\sigma:\lambda)^{-1}\frac{d}{d\lambda}\mathbf{C}(\sigma:\lambda)\right)
=\Tr\left(\mathbf{T}(\sigma,\lambda)^{-1}\frac{d}{d\lambda}\mathbf{T}(\sigma,
\lambda)\right)\Theta_{\sigma,-i\lambda}(\alpha)\\ &-i{\dim}(\sigma)p
\Tr\left(\pi_{\sigma,-i\lambda}(\alpha)J_{\bar{P}_{0}|P_{0}}(\sigma,-i\lambda)^{
-1}\frac{d}{
dz}J_
{\bar{P}_{0}|P_
{0}}(\sigma,-i\lambda)\right).
\end{align*}
\end{kor}
\begin{proof}
We remark that since $\alpha$ is $K$-finite all traces are taken in
finite-dimensional vector spaces. One has
$\dim\left(\boldsymbol{\mathcal{L}}(\sigma)\right)=p\dim(\sigma)$ and by
\eqref{DefJ} one has
\begin{align}\label{J und A}
J_{P_{0}}(\sigma,-i\lambda)^{-1}\frac{d}{d\lambda}J_{P_{0}}(\sigma,-i\lambda)=J_
{\bar{P}_{
0}|P_{0}}(\sigma,-i\lambda)^{-1}\frac{d}{d\lambda}J_{\bar{P}_{0}|P_{0}}(\sigma,
-i\lambda).
\end{align}
Thus the corollary follows from  Proposition \ref{Faktorisierung C-Term}
and the intertwining property of
$\mathbf{I}_{\sigma}$, $\mathbf{I}_{w_0\sigma}$. 
\end{proof}
We shall now determine the logarithmic derivative of the function $\mathbf
{T}(\sigma,\lambda)$. Let $\nu\in\hat{K}$ be a $K$-type of
$\pi_{\sigma,\lambda}$.
Then by section \ref{subsecprs},  $\nu$ occurs with mutliplicity 1 in
$\pi_{\sigma,\lambda}$.
Hence it follows from Schur's Lemma
that
\begin{align}\label{cFnktn}
J_{\bar{P}_{0}|P_{0}}(\sigma,\lambda)|_{(\mathcal{H}^{\sigma})^{\nu}}=c_{\nu}
(\sigma:\lambda)\cdot \Id,
\end{align}
where $c_{\nu}(\sigma:\lambda)\in\C$. 
The function $\lambda\mapsto
c_{\nu}(\sigma:\lambda)$ can be computed explicitly. If
$k_{2}(\nu)e_{2}+\dots+k_{n+1}(\nu)e_{n+1}$ and
$k_{2}(\sigma)e_{2}+\dots+k_{n+1}(\sigma)e_{n+1}$ are  the highest
weight of $\nu$ resp. $\sigma$, by Theorem 8.2 in \cite{Eguchi} one has, taking
the different
parametrization into account:
\begin{align}\label{ErsteGrundglc}
c_{\nu}(\sigma:\lambda)=\alpha(n)\frac{\prod_{j=2}^{n+1}\Gamma(i\lambda-k_{j}
(\sigma
)-\rho_{j})\prod_{j=2}^{n+1}\Gamma(i\lambda+k_{j}(\sigma)+\rho_{j})}{
\prod_{j=2}^{n+1}\Gamma(i\lambda-k_{j}(\nu)-\rho_{j})\prod_{j=2}^{n+1}
\Gamma(i\lambda+k_{j}(\nu)+\rho_{j}+1)},
\end{align}
where $\alpha(n)$ is a constant depending only on $n$. 
Since $\dim\left(\boldsymbol{\mathcal{L}}(\sigma)\right)=p\dim(\sigma)$, 
it follows with
Proposition \ref{Faktorisierung C-Term}, \eqref{J und A} and
\eqref{cFnktn} that for every
$\nu\in\hat{K}$ with $\left[\nu:\sigma\right]\neq 0$ one has  
\begin{align}\label{Spezielle Form S}
\Tr\left(\mathbf{T}(\sigma,i\lambda)^{-1}\frac{d}{dz}\mathbf{T}(\sigma,
i\lambda)\right)=&\frac{1}{\dim(\nu)}\Tr\left(\mathbf{C}(\nu:
\sigma:i\lambda)^{-1}\frac{d}{dz}\mathbf{C}(\nu:\sigma
:i\lambda)\right)\nonumber \\
&+ip\dim(\sigma)c(\nu:\sigma:\lambda)^{-1}\frac{d}{d
\lambda}c(\nu:\sigma:\lambda).
\end{align}
Now for $\sigma\in\hat{M}$ with highest weight
$k_2(\sigma)e_2+\dots+k_{n+1}(\sigma)e_{n+1}$ we let $\nu_\sigma\in\hat{K}$ be
the
representation of $K$
with highest weight $k_2(\sigma)e_2+\dots+\left|k_{n+1}(\sigma)\right|e_{n+1}$.
Then by \cite[Theorem 8.1.4]{Goodman} we have
$\left[\nu_\sigma:\sigma\right]=1$ and thus
using \eqref{ErsteGrundglc} and \eqref{Spezielle Form S} we get
\begin{align}\label{Sc}
\Tr\left(\mathbf{T}(\sigma,i\lambda)^{-1}\frac{d}{dz}\mathbf{T}(\sigma,
i\lambda)\right)=&\frac{1}{\dim(\nu)}\Tr\left(\mathbf{C}
(\nu_\sigma:
\sigma:i\lambda)^{-1}\frac{d}{dz}\mathbf{C}(\nu_\sigma:\sigma
:i\lambda)\right)\nonumber \nonumber\\
&+\sum_{j=2}^{n+1}\frac{p\dim(\sigma)}{i\lambda+\left|k_j(\sigma)\right|+\rho_j}
.
\end{align}
Finally we recall the factorization of the determinant of the $C$-matrix into an
infinite product 
involving its zeroes and poles. 
Let $\sigma\in\hat{M}$ and $\nu\in\hat{K}$ with $\left[\nu:\sigma\right]\neq 0$.
The restrictions of the representations $\pi_{\sigma,\lambda}$ and
$\pi_{w_0\sigma,-\lambda}$ to $K$ are
independent of the parameter $\lambda$ and are unitarily equivalent via the map
$A(w_0):\mathcal{H}^{\sigma}\rightarrow\mathcal{H}^{w_0\sigma}$.  
If we tensor $A(w_0)^{-1}$ with an isometry
$I'(\sigma):\boldsymbol{\mathcal{L}}(w_0\sigma)\rightarrow\boldsymbol{\mathcal{L
}}
(\sigma)$ and use the isomorphisms $\mathbf{I}_\sigma$ and
$\mathbf{I}_{w_0\sigma}$, we obtain an isometry
$I(\sigma)\colon\boldsymbol{\mathcal{E}}(w_0\sigma)\to\boldsymbol{\mathcal{E}}
(\sigma)$ which
maps $\boldsymbol{\mathcal{E}}(w_0\sigma,\nu)$ to
$\boldsymbol{\mathcal{E}}(\sigma,\nu)$ for every $\nu\in\hat{K}$. Moreover by
\eqref{DefJ} and Proposition \ref{Faktorisierung C-Term} for all
$u\in\boldsymbol{\mathcal{L}}(\sigma)$ and all
$\Phi\in(\mathcal{H}^{\sigma})^{\nu}$
we have 
\begin{align*}
I(\sigma)\circ
\mathbf{C}(\nu:\sigma:\lambda)\circ\mathbf{I}_\sigma=\mathbf{I}
_\sigma\circ\left(\left(I'(\sigma)\circ\mathbf
{T}(\sigma,\lambda)\right)\otimes
J_{\bar{P}_0|P_{0}}(\nu,\sigma,-i\lambda)\right).
\end{align*}
Thus using \eqref{cFnktn} it follows that the multiplicity of each pole of
$\det{\left(I(\sigma)\circ\mathbf{C}(\nu:\sigma:\lambda)\right)}$ is
divisible by $\dim(\nu)$.
Let $\{\beta\}$ and $\{\eta\}$ denote the set
of poles of
$\det\left(I(\sigma)\circ\mathbf{C}(\nu:\sigma:\lambda)\right)$ on 
$\left(0,n\right]$ respectively $\{\lambda\in\C\colon \Real(\lambda)<0\}$,
counted with
multiplicity divided by $\dim(\nu)$. Then the set $\{\beta\}$ is finite
and by \cite[Theorem 6.9]{Muller zwei} one has
\begin{align}\label{SpC}
&\frac{1}{\dim(\nu)}\Tr\left(\mathbf{C}(\nu:\sigma:\lambda)^{-1}\frac{d}{
ds}\mathbf{C}(\nu:\sigma:\lambda)\right)\nonumber\\
=&\log{q(\sigma)}+\sum_{\{\beta\}}\left(\frac{
1}{\lambda+\beta}-\frac{1}{\lambda-\beta}\right)+\sum_{\eta}\left(\frac{1}{
\lambda+\overline{\eta}}-\frac{1}{\lambda-\eta}\right),
\end{align}
where $q(\sigma)\in\R^+$ and where the sum converges absolutely. 

\begin{bmrk}\label{RmrkC}
Let $\nu\in\hat{K}$ and $\sigma\in\hat{M}$ with $\left[\nu:\sigma\right]\neq 0$.
Then it follows
from \eqref{FEEs} that $\alpha$ is a pole of $\det{\left(I(\sigma)\circ
\mathbf{C}(\nu:\sigma:s) \right)}$ 
if and only if
$\overline{\alpha}$ is a pole of $\det{\left(I(w_0\sigma)\circ
\mathbf{C}(\nu:\sigma:s) \right)}$ and that the
corresponding orders are equal.
\end{bmrk}

\section{The invariant trace formula}\label{secinv}
\setcounter{equation}{0}
Let $\pi_{\Gamma}$ be the right-regular representation of $G$ on
$L^{2}(\Gamma\backslash G)$. Then there exists an orthogonal
decomposition 
\begin{align}\label{Zerlegung von L2}
L^{2}(\Gamma\backslash G)=L^{2}_{d}(\Gamma\backslash G)\oplus
L^{2}_{c}(\Gamma\backslash G)
\end{align}
of $L^{2}(\Gamma\backslash G)$ into closed $\pi_{\Gamma}$-invariant subspaces.
The restriction of $\pi_{\Gamma}$ to $L^{2}_{d}(\Gamma\backslash G)$ decomposes
into the orthogonal direct sum of irreducible unitary representations of $G$ and
the multiplicity of each irreducible unitary representation of $G$ in this
decomposition is finite. 
On the other hand, by the theory of Eisenstein series, the restriction of
$\pi_{\Gamma}$ to $L^{2}_{c}(\Gamma\backslash G)$ is isomorphic to the direct
integral over all unitary principle-series representations of $G$. These 
results are proved in \cite[sections 1-3]{Warner1}.\\
Now let $\alpha$ be a $K$-finite Schwarz function. Define
an operator
$\pi_{\Gamma}(\alpha)$ on $L^2(\Gamma\backslash G)$ by
\begin{align*}
\pi_{\Gamma}(\alpha)f(x):=\int_{G}\alpha(g)f(xg)dg.
\end{align*}
Then relative to the decompostition \eqref{Zerlegung von L2} one has a splitting
\begin{align*}
\pi_{\Gamma}(\alpha)=\pi_{\Gamma,d}(\alpha)\oplus\pi_{\Gamma,c}(\alpha).
\end{align*}
It easily follows from 
\cite[Theorem 9.1]{Donelly zwei}
that the operator
$\pi_{\Gamma,d}(\alpha)$ is of
trace class. In this section we recall the Selberg trace formula for
$\Tr\left(\pi_{\Gamma,d}(\alpha)\right)$. First we introduce the distributions
involved. Let $I(\alpha):=\vol(X)\alpha(1)$.
By \cite[Theorem 3]{Harish-Chandra2}, the Plancherel theorem can be applied to
$\alpha$. Four
groups of real rank one which do not possess a compact Cartan subgroup it is
stated in \cite[Theorem 13.2]{Knapp}.
Thus if $P_\sigma(z)$ is the Plancherel polynomial with respect to $\sigma$ as
in 
\cite[equation (2.21)]{MP}, one obtains
\begin{align}\label{Idcontr}
I(\alpha)=\vol(X)\sum_{\sigma\in\hat{M}}\int_{\mathbb{R}}{P_{\sigma}
(i\lambda)\Theta_{\sigma,\lambda}(\alpha)}
d\lambda,
\end{align}
where the sum is finite since $\alpha$ is $K$-finite. Next we define the
semisimple contribution by
\begin{align*}
H(\alpha):=\int_{\Gamma\backslash
G}\sum_{\gamma\in\Gamma_{\s}-1}
\alpha(x^{-1}\gamma x)dx.
\end{align*}
Here $\Gamma_{\s}$ are the semisimple elements of $\Gamma$. By \cite[Lemma
8.1]{Warner1} the integral converges absolutely. Its Fourier
transform can be computed as
follows. Let
$\CC(\Gamma)_{\s}$ be the set of semisimple conjugacy classes of
$\Gamma$. For
$\left[\gamma\right]\in\CC(\Gamma)_{\s}-\left[1\right]$ let
$m_\gamma\in M$ and $\ell(\gamma)\in\R^+$ be as in section \ref{secselz}.
Let
$a_\gamma:=\exp{\ell(\gamma)H_1}$. Moreover let $\gamma_0$ be as in section
\ref{secselz}.
Then one puts
\begin{align}\label{hyperbcontr}
L(\gamma,\sigma):=\frac{\overline{\Tr(\sigma)(m_{\gamma})}}
{\det\left(\Id-\Ad(m_\gamma a_\gamma)|_{\bar\nf}\right)}e^{-n\ell(\gamma)}.
\end{align}
Then proceeding as in \cite[Chapter 6]{Wallach} and using \cite[equation
4.6]{Gang} one
obtains
\begin{align}\label{Hyperb}
H(\alpha)=&\sum_{\sigma\in\hat{M}}\sum_{\left[\gamma\right]\in \CC(\Gamma)_{\s}
-\left[1\right]}\frac{l(\gamma_{0})}{2\pi}
L(\gamma,\sigma)\int_{-\infty}^{\infty}{\Theta_{\sigma,\lambda}(\alpha)e^{
-i\ell(\gamma)\lambda}d\lambda},
\end{align}
where the sum is finite since $\alpha$ is $K$-finite. \\
Next let $P\in\mathfrak{P}$ and for every $\eta\in\Gamma\cap N_{P}-\{1\}$ let
$X_{\eta}:=\log{\eta}$. Let
$\left\|\cdot\right\|$ be the norm induced on $\mathfrak{n}_{P}$ by the
restriction of $-\frac{1}{4n}B(\cdot,\theta\cdot)$ to $\mathfrak{n}_{P}$.
Then for $\Real(s)>0$ the Epstein zeta function $\zeta_{P}$ is defined by
$\zeta_{P}(s):=\sum_{\eta\in\Gamma\cap
N_{P}-\{1\}}{\left\|X_{\eta}\right\|}^{-2n(1+s)}$.
By \cite[Chapter 1.4, Theorem 1]{Terras}, this series converges absolutely for
$\Real(s)>0$ and $\zeta_T$ has a
meromorphic continuation to $\C$
with a simple pole at $s=0$.
Let
$R_{P}(\Gamma)$, $C_{P}(\Gamma)$ be the residue resp. the constant term of
$\zeta_{P}$ at $s=0$.
Then by \cite[Chapter 1.4, Theorem 1]{Terras} one has $
R_P(\Gamma)=\frac{\vol(S^{2n-1})}{2n\vol(\Gamma\cap N_P\backslash
N_P)}$.
Now for
$C(\Gamma):=\sum_{P\in\mathfrak{P}}C_{P}(\Gamma)\vol(\Gamma\cap
N_{P}\backslash N_{P})/\vol(S^{2n-1})$ let
\begin{align*}
T(\alpha):=C(\Gamma)\int_{K}\int_
{N}{\alpha(knk^{-1})dn};\quad
T_{P}'(\alpha):=\int_{K}{\int_{N_{P}}{\alpha(kn_{P}k^{
-1}){\log\left\|\log{n_{P
}}\right\|}dn_{P}}dk}.
\end{align*}
Then $T$ and $T_{P'}$ are tempered distributions.
The distributions $T$ is invariant. Applying the Fourier inversion formula and
the Peter-Weyl-Theorem to
equation 10.21 in \cite{Knapp}, one obtains 
the Fourier transform of T as:
\begin{align}\label{Fouriertrafo T}
T(\alpha)=\sum_{\sigma\in\hat{M}}\frac{\dim(\sigma)}{2\pi}C(\Gamma)\int_{
\mathbb{R}}\Theta_{\sigma,\lambda
}(\alpha)d\lambda,
\end{align}
see \cite[Lemma 6.3]{Wallach}.
The distributions $T_{P}'$ are not invariant. However, using the Knapp-Stein
intertwining operators, they can be made invariant as follows.
Let $\epsilon>0$ be
such that 0 is the only possible pole of the operators
$J_{\bar{P}_{0}|P_{0}}(\sigma,z)$, 
$J_{\bar{P}_{0}|P_{0}}(\sigma,z)^{-1}$, $\boldsymbol{T}(\sigma
,z)$, $\boldsymbol{T}(\sigma
,z)^{-1}$ on $\{z\in\C\colon\left|z\right|<2\epsilon\}$ for all
$\sigma\in\hat{M}$
which satisfy $\left[\nu:\sigma\right]\neq 0$, $\nu$ a $K$-type of $\alpha$.
Let $H_{\epsilon}$ be the
half-circle from $-\epsilon$ to $\epsilon$ in the lower half-plane, oriented
counter-clockwise. Let $D_{\epsilon}$ be the path which is the union of
$\left(-\infty,-\epsilon\right]$, $H_{\epsilon}$ and
$\left[\epsilon,\infty\right)$. Let
\begin{align*}
J_{\sigma}(\alpha):=-\frac{ip\dim\sigma}{4\pi
}\int_{D_{\epsilon}}{\Tr\left(J_{\bar{P}_{0}|P_{0}}(\sigma,z)^{-1}\frac{d}{
dz}J_{\bar{P}_{0}|P_{0}}(\sigma,z)\pi_{\sigma,z}(\alpha
)\right)dz}.
\end{align*}
The change of contour is only neccessary if
$J_{\bar{P}_{0}|P_{0}}(\sigma,s)$ has a pole at $0$, i.e. if
$\sigma=w_{0}\sigma$.
Now  we
define a distribution $\mathcal{I}$ by
\begin{align*}
\mathcal{I}(\alpha):=\sum_{P\in\mathfrak{P}}T_{P}'(\alpha)+\sum_{\sigma\in\hat{M
}}J_{\sigma}(\alpha).
\end{align*}
Then by \cite{Hoffmann}, $\mathcal{I}$ is an invariant distribution, see
\cite{MP}. Let
\begin{align}\label{definvS}
\mathcal{S}(\alpha):=\frac{1}{4\pi}\sum_{\sigma\in\hat{M}}\int_{
D_\epsilon}\Tr\left(\boldsymbol{
T}(\sigma,iz)^{-1}\frac{d}{ds}\boldsymbol{T}(\sigma
,iz)\right)\Theta_{\sigma,z}(\alpha)dz,
\end{align}
where $\boldsymbol{
T}(\sigma,iz)$ is as in Proposition \ref{Faktorisierung C-Term}. Then the sum is
finite since $\alpha$ is $K$-finite. By Corollary \ref{Korollar
fur Streuterm} we then have
\begin{align}\label{Gleichung c-Term}
\sum_{\sigma\in\hat{M}}\frac{1}{4\pi}\int_{\R}\Tr\left(\boldsymbol{\pi}_{\Gamma,
\sigma,i\lambda}
(\alpha)\mathbf{C}
(\sigma:i\lambda)^{-1}\frac{d
}
{dz}\mathbf{C}
(\sigma:i\lambda)\right)d\lambda
=\mathcal{S}(\alpha)-J(\alpha). 
\end{align}
Finally, the residual contribution is define by
\begin{align}\label{Def R-Term}
R(\alpha):=\sum_{\substack{\sigma\in\hat{M}\\ \sigma=
w_0\sigma}}-\frac{1}{4}\Tr\left(\mathbf{C}
(\sigma:0)\boldsymbol{\pi}_{\Gamma,\sigma,
0}(\alpha)\right).
\end{align}
This sum is finite since $\alpha$ is $K$-finite.
Using normalized intertwining operators, the Fourier transform of
$R$ can be computed as follows. We use the
notations of section
\ref{secC}. For $\sigma\in\hat{M}$, $\sigma=w_0\sigma$ let
$J_{\bar{P}_{0}|P_{0}}(\sigma,\lambda)$ be as in
\eqref{IntIO}. Then
$J_{\bar{P}_{0}|P_{0}}(\sigma,\lambda)$ might
have a pole at $\lambda=0$. However, if the meromorphic function
$r_{\bar{P}_{0}|P_{0}}(\sigma:\lambda)$ is defined as in \cite[page
113-114]{Hoffmann2},
the map
$R_{P_0}(\sigma:\lambda):=A(w_0)r_{\bar{P}_{0}|P_{0}}(\sigma:\lambda)^{-1}J_{
\bar
{P}_{0}|P_{0}}(\sigma:\lambda)$
is defined and invertible for $\lambda\in\R$, and it satisfies
$R_{P_{0}}(\sigma:0)^{-1}=R_{P_{0}}(\sigma:0)$.
By \cite[Proposition 49, Proposition 53]{Knapp Stein} the representation
$\pi_{\sigma,0}$ is irreducible. Moreover, $R_{P_0}(\sigma:0)$ satisfies
$R_{P_0}(\sigma:0)\circ\pi_{\sigma,0}=\pi_{\sigma,0}\circ
R_{P_0}(\sigma:0)$. Thus by \cite[Corollary 8.13]{Knapp},
$R_{P_0}(\sigma:0)$ is a scalar operator. Thus one has
$\left(R_{P_0}(\sigma:0)\right)^2=\pm\Id$. Let
$\mathbf{S}(\sigma:s):=r_{\bar{P}_{0}|P_{0}}(\sigma:s)\mathbf{T}
(\sigma:s)$, where $\mathbf{T}(\sigma:s)$ is as in Proposition
\ref{Faktorisierung
C-Term}. Then $\mathbf{S}(\sigma:s)$ is a meromorphic
$\Hom(\boldsymbol{\mathcal{L}}(\sigma),\boldsymbol{\mathcal{L}}
(\sigma))$-valued function
and since $\mathbf{C}(\sigma:0)$ is defined and invertible,  it follows
from Proposition \ref{Faktorisierung C-Term} that $\mathbf{S}(\sigma:s)$ is
defined  at $s=0$ and that
$\mathbf{I}_{\sigma}^{-1}\mathbf{C}(\sigma:0)\mathbf{I}_{\sigma}
=\mathbf{S}(\sigma:0)\otimes R_{P_0}(\sigma:0)$, where $\mathbf{I}_{\sigma}$ is
as in section \ref{secC}.
Using the functional equation \eqref{FEEs} one obtains
$\mathbf{S}(\sigma:0)^*=\mathbf{S}(\sigma:0)$ and $
\mathbf{S}(\sigma:0)^{-1}=\mathbf{S}(\sigma:0)$.
Hence $\mathbf{S}(\sigma:0)$ is diagonalizable with eigenvalues $\pm 1$. Using 
the intertwining property of $\boldsymbol{I}_\sigma$, it follows that there
exist natural numbers $c_1(\sigma), c_2(\sigma)$
with
$c_1(\sigma)+c_2(\sigma)=p\dim(\sigma)$ such that one has
\begin{align}\label{Resfrml}
-\frac{1}{4}\Tr\left(\mathbf{C}
(\sigma:0)\boldsymbol{\pi}_{\Gamma,\sigma,
0}(\alpha)\right)=\frac{c_1(\sigma)-c_2(\sigma)}{4}\Theta_{\sigma,0}(\alpha)
\end{align}
for every $K$-finite Schwarz function $\alpha$.  
We can now state the invariant trace formula.
\begin{thrm}\label{Spurf}
Let $\alpha$ be a $K$-finite Schwarz function. 
Then one has
\begin{align*}
\Tr\left(\pi_{\Gamma,d}(\alpha)\right)=I(\alpha)+H(\alpha
)+T(\alpha)+\mathcal{I}(\alpha)+R(\alpha)+\mathcal{S}(\alpha).
\end{align*}
\end{thrm}
\begin{proof}
This theorem is a special case of the invariant trace formula stated in
\cite[Theorem 6.4]{Hoffmann2}. It follows if one combines \cite[Theorem
8.4]{Warner1},
the
Theorem on
page 299 in \cite{Osborne}, the explicit form of $R_P(\Gamma)$ and
\eqref{Gleichung c-Term}. Here one has to take into account that
our normalizations are different from those of \cite{Osborne}.
\end{proof}
The Fourier transform of the distribution $\mathcal{I}$ was computed in
\cite{Hoffmann}. We shall now state his result.
Let $S(\mathfrak{b}_{\C})$ be the symmetric algebra of $\mathfrak{b}_{\C}$. 
Define $\Pi\in S(\mathfrak{b}_{\C})$ by
$\Pi:=\prod_{\alpha\in\Delta^{+}(\mathfrak{m}_{\C},\mathfrak{b}_{\C})}
H_{\alpha}$.
For $\sigma\in\hat{M}$ with highest weight
$k_{2}(\sigma)e_{2}+\dots+k_{n+1}(\sigma)e_{n+1}$ and
$\lambda\in\mathbb{R}$ define
$\lambda_{\sigma}\in(\mathfrak{h})_{\C}^{*}$ by 
$\lambda_{\sigma}:=i\lambda
e_{1}+\sum_{j=2}^{n+1}(k_{j}(\sigma)+\rho_{j})e_{j}$.
We will denote by $\left<\cdot,\cdot\right>$ the symmetric
bilinear form on $\mathfrak{h}_{\C}^{*}$ induced by the Killing form.
Then for
$\alpha\in\Delta^{+}(\mathfrak{g}_{\C},\mathfrak{h}_{\C})$
we denote by
$s_{\alpha}:\mathfrak{h}_{\C}^{*}\rightarrow \mathfrak{h}_{
\C}^{*}$ the reflection
$s_{\alpha}(x)=x-2\frac{\left<x,\alpha\right>}{\left<\alpha,\alpha\right>}
\alpha$. 
Now the Fourier transform of $\mathcal{I}$ is computed as follows. 
\begin{thrm}\label{Hoffmanns Theorem}
For every $K$-finite $\alpha\in \mathcal{C}^{2}(G)$ one has
\begin{align*}
\mathcal{I}(\alpha)=\frac{p}{4\pi}\sum_{\sigma\in\hat{M}}\int_{\mathbb{R
}}{\Omega(\check{\sigma},-\lambda)\Theta_{\sigma,\lambda}(\alpha)d\lambda},
\end{align*}
where the function $\Omega(\sigma,\lambda)$ is given by
\begin{align*}
\Omega(\sigma,\lambda):=-2\dim(\sigma)\gamma-\frac{1}{2}\sum_{
\alpha\in\Delta^{+}(\mathfrak{g}_\C,\mathfrak{a}_\C)}\frac{\Pi(s_{\alpha}
\lambda_{\sigma})
}{\Pi(\rho_{M})}\left(\psi(1+\lambda_{\sigma}(H_{\alpha}))+\psi(1-\lambda_{
\sigma}(H_{\alpha})\right).
\end{align*}
Here $\psi$ denotes the digamma function and $\gamma$ denotes the
Euler-Mascheroni constant. Moreover $\check{\sigma}$ denotes the contragredient
representation of $\sigma$.
\end{thrm}
\begin{proof}
This follows from \cite[Theorem 5]{Hoffmann}, \cite[Theorem 6]{Hoffmann},
\cite[Corollary on page 96]{Hoffmann}.
\end{proof}
In order to show that the residues
of the logarithmic derivative of the symmetrized Selberg 
zeta function are integral, we shall now take a closer look at the
functions $\Omega(\sigma,\lambda)$. For $\sigma$ of highest
weight
$k_2(\sigma)e_2+\dots+k_{n+1}(\sigma)e_{n+1}$ and $j=2,\dots,n+1$ we let 
\begin{align*}
P_j(\sigma,\lambda)=\frac{\Pi(s_{e_{1}+e_{j}}\lambda_{\sigma})}{\Pi(\rho_{M}
)}.
\end{align*}
Then it was shown in \cite[equation 6.23 ]{MP} that
\begin{align}\label{GLP}
P_j(\sigma,\lambda)=\dim(\sigma)\prod_{\substack{p=2\\ p\neq
j}}^{n+1}\frac{-\lambda^2-(k_p(\sigma)+\rho_p)^2}{
(k_j(\sigma)+\rho_j)^2-(k_p(\sigma)-\rho_p)^2}.
\end{align}
Thus
$P_{j}(\sigma,\lambda)$ is an even polynomial in $\lambda$ of degree $2n-2$.
Now the following lemma holds.
\begin{lem}\label{Lemint}
Let $j=2,\dots,n+1$ and assume that all $k_{j}(\sigma)$ are integers
with $k_{n+1}(\sigma)\geq 0$. Let
$l\in\mathbb{N}$ with $k_{n+1}(\sigma)\leq l\leq k_{j}(\sigma)+\rho_{j}$.
Then $P_j(\sigma,il)$ is integral.
\end{lem}
\begin{proof}
Suppose first that $l=k_{n+1}(\sigma)$. If $j<n+1$, 
one has
$P_j(\sigma,il)=0$  and if $j=n+1$, one has
$P_j(\sigma,il)=\dim(\sigma)$ by \eqref{GLP}.\\
Now suppose that $l>k_{n+1}(\sigma)$. Then there exists a minimal
$\nu\in\{2,\dots,n\}$ such
that $
l-\rho_{\nu}\geq k_{\nu+1}(\sigma)$. We have $\nu\geq j$.
If $\nu>2$ we have
$l-\rho_{\nu-1}<k_{\nu}(\sigma)$. Moreover, if $\nu>2$ with
$l-\rho_{\nu-1}=k_{\nu}(\sigma)-1$ then
$l=k_{\nu}(\sigma)+\rho_{\nu}$ and so in this case we have 
$P_j(\sigma,il)=\dim(\sigma)$
for $\nu=j$ and
$P_j(\sigma,il)=0$
for $\nu>j$ by \eqref{GLP}. \\
Thus it remains to consider the case that either $\nu=2$ or that for $\nu>2$ one
has
$k_{\nu}(\sigma)-1\geq l-\rho_{\nu-1}+1=l-\rho_{\nu}$.
In this case we define $\Lambda^\prime\in\mathfrak{b}_{\C}^{*}$ by
\begin{align*}
\Lambda^\prime:&=\sum_{2\leq i<j}k_i(\sigma)e_i+\sum_{j<
i\leq\nu}\left(k_i(\sigma)-1\right)e_{i-1}+
(l-\rho_{\nu})e_{\nu}+
\sum_{\nu<i\leq n+1}k_i(\sigma)e_{i}.
\end{align*}
Then $\Lambda^\prime$
is the highest weight $\Lambda(\sigma^\prime)$ of a representation
$\sigma^\prime$ of $M$. By \cite[equation 6.19]{MP}, the restriction of 
$s_{e_1+e_j}\lambda_{\sigma}|_{\lambda=il}$ to $\bL^*_\C$ coincides with
$\Lambda(\sigma'+\rho_M)$ up to a permutation 
of the $e_2,\dots,e_n$. Let $\xi\in\bL_\C^*$,
$\xi=\xi_2e_2+\dots+\xi_{n+1}e_{n+1}$. Then by \cite[equation 6.18]{MP}, if
$\tau$ is any permutation of
$\{2,\dots,n+1\}$, for
$\xi_{\tau}:=\xi_{2}e_{\tau(2)}+\dots+\xi_{n+1}e_{\tau(n+1)}$, one has
$\Pi(\xi_{\tau})=\pm\Pi(\xi)$. Applying the Weyl dimension formula \cite[Theorem
7.19]{Goodman} one obtains
\begin{align*}
\dim(\sigma^\prime)=\frac{\Pi(\Lambda(\sigma')+\rho_{M})}{\Pi(\rho_M)}
=\pm P_j(\sigma,il).
\end{align*}
This proves the lemma. 
\end{proof}
We can now state our main result about the function $\Omega(\sigma,\lambda)$. 
\begin{prop}\label{Para}
Let $\sigma\in\hat{M}$ with highest weight
$k_{2}(\sigma)e_{2}+\dots+k_{n+1}(\sigma)e_{n+1}$. 
Assume that
all
$k_{j}(\sigma)$ are integers. Let
$m_{0}:=\left|k_{n+1}(\sigma)\right|$.  
Then there exists an even polynomial
$Q(\sigma,\lambda)$ of degree $\leq 2n-4$ and for every $j=2,\dots,n+1$, every
$l$
with $m_0\leq l< \left|k_{j}(\sigma)\right|+\rho_{j}$ there exist integers
$c_{j,l}(\sigma)$ such that
\begin{align*}
\Omega(\sigma,\lambda)=&-\dim(\sigma
)\left(2\gamma+\psi(1+i\lambda)+\psi(1-i\lambda)+\sum_{\substack{1\leq l<
m_0}}\frac{2l}{\lambda^2+l^2}\right)\\
&-\sum_{j=2}^{n+1}\sum_{\substack{m_0\leq
l<\\\left|k_{j}(\sigma)\right|+\rho_{j}}}c_{j,l}
(\sigma)\frac{2l}{\lambda^2+l^2}
-\sum_{j=2}^{n+1}\dim(\sigma)\frac{\left|k_{j}
(\sigma)\right|+\rho_{j}}{(\left|k_{j}
(\sigma)\right|+\rho_{j})^2+\lambda^2}-Q(\sigma,\lambda).
\end{align*}
For every $\sigma\in\hat{M}$ one has
$\Omega(\sigma,\lambda)=\Omega(w_0\sigma,\lambda)=\Omega(\check{\sigma},
\lambda)$ and $Q(\sigma,\lambda)=Q(w_0\sigma,\lambda)$. 
\end{prop}
\begin{proof}
For $j=2,\dots,n+1$ and 
$l\in\N$ with $m_0\leq l< \left|k_{j}(\sigma)\right|+\rho_{j}$ we let
$c_{j,l}(\sigma):=P_j(\sigma,il)$.
Then $c_{j,l}(\sigma)$ is integral by Lemma \ref{Lemint}.
Thus if $|k_{n+1}(\sigma)|>0$, the proposition follows from \cite[Proposition
6.4]{MP}. If 
$k_{n+1}(\sigma)=0$, the proof is exactly the same.
\end{proof}
We remark that if all $k_j(\sigma)$ are half-integral, a statement analogous to 
Proposition \ref{Para} holds. For the proof one proceeds in the same way as
above.

\section{Bochner Laplace operators}\label{secBLO}
\setcounter{equation}{0}
Let $(\nu,V_\nu)\in\hat{K}$ and let $\tilde{E}_\nu:=G\times_\nu V_\nu$ be the
associated
homogeneous 
vector bundle over $\widetilde X$ equipped with the metric induced from $V_\nu$.
Then if one lets
\begin{align}\label{defsect}
C^{\infty}(G,\nu):=\{f:G\rightarrow V_{\nu}\colon f\in C^\infty,\;
f(gk)=\nu(k^{-1})f(g),\,\,\forall g\in G, \,\forall k\in K\},
\end{align}
the smooth sections of $\tilde{E}_\nu$ are canonically isomorphic to
$C^\infty(G,\nu)$. A 
corresponding isometry also holds for the $L^2$-sections of $\tilde{E}_\nu$ and
the  space $L^2(G,\nu)$, which is defined 
analogously to \eqref{defsect}. Let $\tilde{A}_\nu$ be the 
differential operator which is induced by $-\Omega$ on $C^\infty(G,\nu)$. Then
by the same argument as in \cite[section 4]{MP} $\tilde{A}_\nu$ with domain
$C_c^\infty(G,\nu)$, the compactly supported elements of $C^\infty(G,\nu)$, is
bounded from below and
essentially selfadjoint. Its selfadjoint closure will be denoted by
$\tilde{A}_\nu$ too. Let $e^{-t\tilde{A}_\nu}$ be the heat-semigroup of
$\tilde{A}_\nu$ on 
$L^2(G,\nu)$. Then there exists a function 
\begin{align}\label{DefH}
{H}^{\nu}_{t}:G\longrightarrow {\rm{End}}(V_{\nu});\quad
{H}^{\nu}_{t}(k^{-1}gk')=\nu(k)^{-1}\circ {H}^{\nu}_{t}(g)\circ\nu(k'),
\:\forall k,k'\in K, \forall g\in G
\end{align}
such that
\begin{align}\label{Eqheat}
(e^{-t\tilde{A}_{\nu}}\phi)(g)=\int_{G}{{H}^{\nu}_{t}(g^{-1}g')\phi(g')dg'}
,
\quad\phi\in  L^{2}(G,\nu),\quad g\in G,
\end{align}
see \cite[section 4]{MP}. By the  arguments of \cite[Proposition 2.4]{Barbasch},
 $H^\nu_t$ belongs to
all 
Harish-Chandra Schwarz spaces 
$(\mathcal{C}^{q}(G)\otimes {\rm{End}}(V_{\nu}))$, $q>0$. We put
$h^{\nu}_{t}(g):=\tr{H}^{\nu}_{t}(g)$, where $\tr$ denotes the trace in
$\End{V_\nu}$. 
\\
Now we pass to the quotient $X=\Gamma\backslash\widetilde{X}$. 
Let $E_\nu:=\Gamma\backslash \tilde{E}_\nu$ be the locally homogeneous bundle 
over $X$. Let $C_c^\infty(\Gamma\backslash G,\nu)$ resp. $L^2(\Gamma\backslash
G,\nu)$ denote the $\Gamma$-left invariant 
elements of $C_c^\infty(G,\nu)$ resp.
$L^2(G,\nu)$ .
Then $A_\nu$ with domain
$C_c^\infty(\Gamma\backslash G,\nu)$ is essentially selfadjoint and bounded from
below, see \cite[section 4]{MP}. Its 
selfadjoint closure will be denoted by $A_\nu$ too. The decomposition
\eqref{Zerlegung von L2} induces a decomposition of
$L^2(\Gamma\backslash G,\nu)\cong \left(L^2(\Gamma\backslash G,\nu)\otimes
V_\nu\right)^K$ as $
L^2(\Gamma\backslash
G,\nu)=L^2_d(\Gamma\backslash
G,\nu)\oplus L^2_c(\Gamma\backslash
G,\nu)$.
This decomposition is invariant under $A_\nu$ in the sense of unbounded
operators.
Let $A_\nu^d$ denote the restriction of $A_\nu$ to $L^2_d(\Gamma\backslash
G,\nu)$ in the sense of unbounded operators. Since $\pi_{\Gamma,d}$ decomposes
discretely into a direct sum of irreducible unitary representations of $G$,
there 
exists a subset $J\subset\N$, a sequence $\lambda_j$, $j\in J$ of real numbers 
and an orthonormal base $\phi_j$, $j\in J$ of $L^2_d(\Gamma\backslash G,\nu)$
such that
$A_\nu \phi_j=\lambda_j\phi_j$ for every $j\in J$. The set $J$ may be finite.
For $\lambda\in\left[0,\infty\right)$
let $N(\lambda):=\#\{j\in J\colon\lambda_{j}\leq\lambda\}$.
By \cite[Theorem 9.1]{Donelly zwei},
there exists a constant $C$ such that
\begin{align}\label{Wlaw}
N(\lambda)\leq C(1+\lambda)^{\frac{d}{2}}.
\end{align} 
for every $\lambda$. Hence the operator $e^{-tA_{\nu}^{d}}$ is of trace class
and one has
\begin{align}\label{summed}
\Tr\pi_{\Gamma,d}(h_t^\nu)=\Tr(e^{-tA_{\nu}^{d}})=\sum_{j}{e^{
-t\lambda_{j}}}.
\end{align}
Using 
the methods of \cite{Barbasch}, it was shown in \cite[section 4]{MP} that
\begin{align}\label{fouriertrf2}
\Theta_{\sigma,\lambda}(h_t^\nu)=e^{t(c(\sigma)-\lambda^{2})}\left[
\nu:\sigma\right].
\end{align}
for every $\sigma\in\hat{M}$. Here $c(\sigma)$ is as in section
\ref{subsecprs}.

\section{The symmetric Selberg zeta
function}\label{secsymzeta}
\setcounter{equation}{0}
For $s\in\C$, $\Real(s)>2n$ we define the symmetric Selberg zeta function
$S(s,\sigma)$ 
by $S(s,\sigma):=Z(s,\sigma)$ if $\sigma=w_0\sigma$ and by
by $S(s,\sigma):=Z(s,\sigma)+Z(s,w_0\sigma)$ if $\sigma\neq w_0$.
In this section we want to study the function $S(s,\sigma)$.\\
By Proposition \ref{branching} there exist unique integers 
$m_{\nu}(\sigma)\in\{-1,0,1\}$ which are zero except for finitely many
$\nu\in\hat{K}$ 
such that
\begin{align}\label{Def der m} 
\sum_{\nu\in\hat{K}}m_{\nu}(\sigma)\iota^{*}\nu=\sigma,\:\text{if
$\sigma=w_{0}\sigma$};
\quad \sum_{\nu\in\hat{K}}m_{\nu}(\sigma)\iota^{*}\nu=\sigma+w_{0}\sigma,\:
\text{if
$\sigma\neq w_{0}\sigma$}.
\end{align}
In the notations of section \ref{secBLO}, we define vector bundles
$\tilde{E}(\sigma)$ and $E(\sigma)$ over $\widetilde X$ resp. X by 
\begin{align}\label{Definition des Hilfsbundels}
\tilde E(\sigma):=\bigoplus_{\substack{\nu\in\hat{K}\\ m_{\nu}(\sigma)\neq
0}}\tilde
E_{\nu};\qquad
E(\sigma):=\bigoplus_{\substack{\nu\in\hat{K}\\ m_{\nu}(\sigma)\neq 0}} E_{\nu}.
\end{align}
Then $\tilde E(\sigma)$ and $E(\sigma)$ admit gradings $\tilde
E(\sigma)=\tilde{E}^{+}(\sigma)\oplus \tilde{E}^{-}(\sigma)$ and
$E(\sigma)=E^{+}(\sigma)\oplus E^{-}(\sigma)$
defined by the sign of $m_{\nu}(\sigma)$. For $\nu\in\hat{K}$ let
$A_{\nu}$ and $A_\nu^d$ be as in section \ref{secBLO} and define
\begin{align*}
A(\sigma):=\bigoplus_{\substack{\nu\in\hat{K}\\ m_{\nu}(\sigma)\neq
0}}A_{\nu}+c(\sigma);
\quad A(\sigma)_d:=\bigoplus_{\substack{\nu\in\hat{K}\\ m_\nu(\sigma)\neq
0}}A_{\nu}^d+c(\sigma),
\end{align*}
where $c(\sigma)$ is as in section \ref{subsecprs}. 
Let $\tilde{A}(\sigma)$ be the lift of 
$A(\sigma)$ to $\tilde{E}(\sigma)$.
Let 
\begin{align}\label{Defh}
h^{\sigma}_{t}(g):=e^{-tc(\sigma)}\sum_{\nu\in\hat{K}}
m_{\nu}(\sigma)h^{\nu}_{t}(g),
\end{align}

where $h^{\nu}_{t}$ is as in the previous section. Then by \eqref{Def der m} 
and \eqref{fouriertrf2}, for a principal series representation 
$\pi_{\sigma',\lambda}$, $\sigma'\in\hat{M}$, $\lambda\in\mathbb{R}$ we have
\begin{align}\label{FTSup}
\Theta_{\sigma',\lambda}(h^{\sigma}_{t})=e^{-t\lambda^{2}} \quad
\text{for $\sigma'\in\{\sigma, w_{0}\sigma\}$};\qquad
\Theta_{\sigma',\lambda}(h^{\sigma}_{t})=0, \quad\text{otherwise}. 
\end{align}
Next we state a generalized resolvent formula which is due to Bunke and Olbrich.

\begin{prop}\label{Resolvent}
Let $s_{1},\dots,s_{N}$ be complex numbers with $s_i^2\neq s_{i'}^2$ for $i\neq
i'$. Then for every
$z\in\mathbb{C}-\{-s_{1}^{2},\dots,-s_{N}^{2}\}$ one has
\begin{align*}
\sum_{i=1}^{N}\frac{1}{s_{i}^{2}+z}\prod_{\substack{i'=1\\ i'\neq
i}}^{N}\frac{1}{s_{i'}^{2}-s_{i}^{2}}=\prod_{i=1}^{N}\frac{
1}{s_{i}^{2}+z}.
\end{align*}
\end{prop}
\begin{proof}
This is proved by Bunke and Olbrich, \cite[Lemma 3.5]{Bunke}.
\end{proof}
Since the operators $A_\nu$ are bounded from below, there
exists a
$\lambda_0(\sigma)\in\R$
such that $A(\sigma)\geq\lambda_0(\sigma)$. Let $\lambda_{0}<\lambda_{1}<\dots$
be the sequence of eigenvalues of $A(\sigma)$. This sequence
might be either a finite or an infinte. For each $\lambda_{k}$ let
$\mathcal{E}(\lambda_{k})$ be the eigenspace of $A(\sigma)$ with eigenvalue
$\lambda_{k}$ and define its graded dimension by
$m_{\s}(\lambda_{k},\sigma):=\dim_{\gr}\mathcal{E}(\lambda_{k})$.
Now fix $N\in\N$ with $N>d/2$ and choose
distinct $s_1,\dots,s_{N}\in\C$, such
that $\Real(s_{i})>2n$ and such that $\Real(s_i^2)>\max\{0,-\lambda_0(\sigma)\}$
for all
i.
Then by Proposition \ref{Resolvent} we have
\begin{align}\label{EqPrTr}
\sum_{k}m_{\s}(\lambda_{k},\sigma)\prod_{i=1}^{N}\frac{1}{\lambda_{k}+s_{
i}^{2}}=\int_{0}^{\infty}\sum_{i=1}^{N}\Tr_{\s}e^{-t\left(A(\sigma)_{d}+s_{i}
^{2}\right)}\prod_{\substack{i'=1\\ i'\neq
i}}^{N}\frac{1}{s_{i'}^{2}-s_{i}^{2}}
dt,
\end{align}
where the super-trace is taken with respect to the grading defined above. The
sum on the left hand 
side of \eqref{EqPrTr} converges absolutely by \eqref{Wlaw}.
We compute the right hand side of \eqref{EqPrTr} using the \eqref{summed} and
the invariant trace
formula stated in
Theorem \ref{Spurf}. From now on we assume that in the weight $\Lambda(\sigma)$
of 
$\sigma$ as in \eqref{RepM} all $k_i(\sigma)$ are integers. If 
all $k_i(\sigma)$ are half-integral, one can proceed in the same way by making 
the appropriate modifications in the treatment of the distribution
$\mathcal{I}$.
To save notation we let $\epsilon(\sigma)=1$, if $\sigma=w_{0}\sigma$ and
$\epsilon(\sigma)=2$, if $\sigma\neq w_{0}\sigma$.\\
We begin with the identity contribution.
The polynomial $P_\sigma(\lambda)$ is even and of degree $d-1$ and by
\cite[section 2.8]{MP} it
satisfies $P_\sigma(\lambda)=P_{w_0\sigma}(\lambda)$.
Thus by \eqref{Idcontr} and \eqref{FTSup} one has
\begin{align}\label{Icontr}
\int_{0}^{\infty}{\sum_{i=1}^{N}I(h^{\sigma}_{t})e^{-ts_{i}^{2}}\prod_{\substack
{i'=1\\ i'\neq
i}}^{N}\frac{1}{s_{i'}^{2}-s_{i}^{2}}dt}
=\epsilon(\sigma)\pi{\rm{vol}}(X)\sum_{i=1}^{N}P_{\sigma}(
s_{i})\frac{1}{s_{i}}\prod_{\substack{i'=1\\ i'\neq
i}}^{N}\frac{1}{s_{i'}^{2}-s_{i}^{2}}.
\end{align}
Next we come to the hyperbolic contribution. For $\gamma\in\CC(\Gamma)_{\s}$ we
let
$L_{{\rm{sym}}}(\gamma,\sigma):=L(\gamma,\sigma)$, if
$\sigma=w_{0}\sigma$ and
$L_{{\rm{sym}}}(\gamma,\sigma):=L(\gamma,\sigma)+L(\gamma,w_{0}\sigma)$, 
if $\sigma\neq w_0\sigma$.
Here $L(\gamma,\sigma)$ is as in \eqref{hyperbcontr}.
Then using \eqref{Hyperb} and \eqref{FTSup} we obtain
\begin{align*}
&\int_{0}^{\infty}{\sum_{i=1}^{N}H(h^{\sigma}_{t})e^{-ts_{i}^{2}}\prod_{
\substack{i'=1\\ i'\neq
i}}^{N}\frac{1}{s_{i'}^{2}-s_{i}^{2}}dt}
=\sum_{i=1}^{N}\sum_{[\gamma]\in\CC(\Gamma)_{\s}-\left[
1\right]}
\ell(\gamma_{0
})L_{{\rm{sym}}}(\gamma,\sigma)e^{-s_il(\gamma)}\frac{1}{2s_{i}}\prod_{\substack
{i'=1\\ i'\neq
i}}^{N}\frac{1}{s_{i'}^{2}-s_{i}^{2}}.
\end{align*}
Since $\Real(s_i)>2n$ for every i, the sum on the right hand side converges
absolutely by equation \eqref{LS} and equation \eqref{estdet}. By \cite[section
3.2.5]{Goodman} if $n$ is
even we have $\check{\sigma}=\sigma$ 
and if $n$ is odd we have
$\check{\sigma}=w_0\sigma$ for every $\sigma\in\hat{M}$.
Thus, applying \eqref{LogSZF} we obtain
\begin{align}\label{Scontr}
\int_{0}^{\infty}{\sum_{i=1}^{N}H(h^{\sigma}_{t})e^{-ts_{i}^{2}}\prod_{\substack
{i'=1\\ i'\neq
i}}^{N}\frac{1}{s_{i'}^{2}-s_{i}^{2}}dt}
=\sum_{i=1}^{N}\frac{d}{ds}\biggr|_{s=s_i}\log{S(s,\sigma)
}\frac{1}{2s_{i}}\prod_{\substack{i'=1\\i'\neq
i}}^{N}\frac{1}{s_{i'}^{2}-s_{i}^{2}}.
\end{align}
For the distribution $T$ it follows from 
\eqref{Fouriertrafo T} and \eqref{FTSup} 
that
\begin{align}\label{Tcontr}
\int_{0}^{\infty}{\sum_{i=1}^{N}T(h^{\sigma}_{t})e^{-ts_{i}^{2}}\prod_{\substack
{i'=1\\ i'\neq
i}}^{N}\frac{1}{s_{i'}^{2}-s_{i}^{2}}dt}
=\frac{\epsilon(\sigma)\dim(\sigma)}{2}C(\Gamma)\sum_{i=1}^{N}\frac{1}{s_{i}}
\prod_{
\substack{i'=1\\
i'\neq i}}^{N}\frac{1}{s_{i'}^{2}-s_{i}^{2}}.
\end{align}
Using Theorem \ref{Hoffmanns Theorem}, Proposition \ref{Para},
and equation \eqref{FTSup} we compute:
\begin{align*}
&\int_{0}^{\infty}{\sum_{i=1}^{N}\mathcal{I}(h^{
\sigma}_{t})e^{-ts_{i}^{2}}\prod_{\substack{i'=1\\ i'\neq
i}}^{N}\frac{1}{s_{i'}^{2}-s_{i}^{2}}dt}=-\frac{p\epsilon(\sigma)}{2}\sum_{j=2}^
{n+1}\sum_{\substack{m_0\leq
l\\<\left|k_{j}
(\sigma)\right|+\rho_
j}}
c_{j,l}(\sigma)\sum_{i=1}^{N}\frac{1}{l+s_{i}}\frac{1}{s_{i}}\prod_{\substack{
i'=1\\ i'\neq i}}^{N}\frac{1}{s_{i'}^{2}-s_{i}^{2}}\\
&-\frac{p\epsilon(\sigma)\dim(\sigma)}{2}\sum_{i=1}^{N}
\left(\gamma+\psi(1+s_{i})+\sum_{\substack{1\leq
l<
m_0}}\frac{1}{l+s_i}\right)\frac{1}{s_{i}}\prod_{
\substack{i'=1\\ i'\neq
i}}^{N}\frac{1}{s_{i'}^{2}-s_{i}^{2}}\\
&-\frac{1}{4}\sum_{\substack{j=2\\
|k_j(\sigma)|+\rho_j>0}}^{n+1}\sum_{i=1}^{N}\frac{p\epsilon(\sigma)\dim(\sigma)}
{
\left|k_{j}(\sigma
)\right|+\rho_{j}+s_{i}}\frac{1}{
s_{i}
}\prod_{\substack{i'=1\\ i'\neq
i}}^{N}\frac{1}{s_{i'}^{2}-s_{i}^{2}}
-\frac{1}{4}\sum_{i=1
}^{N}Q(\sigma,is_{i})\frac{p\epsilon(\sigma)}{s_{i}}\prod_{\substack
{i'=1\\ i'\neq
i}}^{N}\frac{1}{s_{i'}^{2}-s_{i}^{2}}.
\end{align*}
Next  we treat the contribution of the distribution $\mathcal{S}$ defined in
\eqref{definvS}. For
$\sigma\in\hat{M}$ let $\nu_\sigma\in\hat{K}$
be as in section \ref{secC} and let
$\{\beta(\sigma)\}$
and $\{\eta(\sigma)\}$ denote the poles of
$\det{\left(I(\sigma)\circ
\mathbf{C}(\nu_\sigma:\sigma:s) \right)}$ 
on $\left(0,n\right]$ and on $\{s\in\C\colon\Real(s)<0\}$, counted with
multiplicity
divided by $\dim(\nu_\sigma)$ as in section \ref{secC}. Then
by \eqref{FEEs}, \eqref{Sc},
\eqref{SpC} and \eqref{FTSup}
one has
\begin{align*}
&\int_{0}^{\infty}{\sum_{i=1}^{N}\mathcal{S}
(h^{\sigma}_{t})e^{-ts_{i}^{2}}\prod_{\substack{i'=1\\ i'\neq
i}}^{N}\frac{1}{s_{i'}^{2}-s_{i}^{2}}dt}
=\frac{\epsilon(\sigma)}{4}\sum_{\{\eta(\sigma)\}}\sum_{i=1}^{N}\left(\frac{1}{
\eta(\sigma)-s_{i}}+\frac{1}{\bar{\eta}(\sigma)-s_{i}}\right)\frac{1}{s_{i}}
\prod_{
\substack{i'=1\\ i\neq
i'}}^{N}\frac{1}{s_{i'}^{2}-s_{i}^{2}}\\
&+\frac{\epsilon(\sigma)}{2}\sum_{\{ \beta(\sigma)\}}\sum_{
i=1}^{N}\frac{1}{s_{i}
+\beta(\sigma)}\frac{1}{s_{i}}\prod_{\substack{i'=1\\ i'\neq
i}}^{N}\frac{1}{s_{i'}^{2}-s_{i}^{2}}+\frac{\epsilon(\sigma)}{4}\sum_{
i=1}^{N}\log{q(\sigma)}\frac{1}{s_{i}}\prod_{\substack{i'=1\\ i'\neq
i}}^{N}\frac{1}{s_{i'}^{2}-s_{i}^{2}}\\
&+\frac{\epsilon(\sigma)p\dim(\sigma)}{4}\sum_{j=2}^{n+1}\sum_{
i=1}^{N}\frac{
1}{\left|k_j(\sigma)\right|+\rho_j+s_i}\frac{1}{s_{i}}\prod_{
\substack{i'=1\\ i'\neq
i}}^{N}\frac{1}{s_{i'}^{2}-s_{i}^{2}}.
\end{align*}
Finally, the residual contribution is nonzero only if $\sigma=w_{0}\sigma$ and
using
\eqref{Resfrml} and \eqref{FTSup} it follows that in this
case we have
\begin{align*}
\int_{0}^{\infty}{\sum_{i=1}^{N}R(h^{\sigma}_{t}
)e^{-ts_{i}^{2}}\prod_{\substack{i'=1\\ i'\neq
i}}^{N}\frac{1}{s_{i'}^{2}-s_{i}^{2}}dt}=\frac{c_{1}(\sigma)-c_{2}(\sigma)}{4}
\sum_{i=1}^{N}\frac{1}{s_{i}^{
2}}\cdot\prod_{i'=1}^{N}\frac{1}{s_{i'}^{2}-s_{i}^{2}}.
\end{align*}
Now for $\sigma\neq w_{0}\sigma$
put
$K(\sigma,s)=0$ and for $\sigma=w_0\sigma$ put
$K(\sigma,s)=-\frac{c_{1}(\sigma)}{s}$. Moreover let $c_{\Gamma}
(\sigma):=\epsilon(\sigma)\left(\dim(\sigma)C(\Gamma)-\dim(\sigma)\gamma
p\right)$. 
Then the computations of this section can be summarized in the following
proposition. 
\begin{prop}\label{PropSymS}
For $s\in\C$ with $\Real(s)>2n$ define
\begin{align*}
\Xi(s,\sigma)
:=&\exp\left(2\pi{\rm{vol}}(X)\epsilon(\sigma)\int_{0}^{s}{P_{\sigma}
(r)dr}-\epsilon(\sigma)\frac{p}{2}\int_{0}^{s}{Q(\sigma,
ir)dr
}+sc_{\Gamma}(\sigma)\right)\\
&\cdot\left(\Gamma(1+s)\right)^{-p\epsilon(\sigma)\dim(\sigma)}\cdot
S(s,\sigma).
\end{align*}
Then there exist polynomials $C(\lambda_k;s)$ 
and $C(\sigma;s)$  in $s$ which are of degree
at most $2(N-2)$ such that
the logarithmic derivative of $\Xi$ satisfies
\begin{align*}
\frac{\Xi'(s,\sigma)}{\Xi(s,\sigma)}
=&2s\sum_{k}m_{\s}(\lambda_{k},
\sigma)\left(\frac{1}{s^{2}+\lambda_{k}}+C(\lambda_k;s
)\right)+\sum_{1\leq l< m_0}p\epsilon(\sigma)\dim(\sigma)
\frac{1}{l+s}\\ &+\sum_{j=2}^{n+1}\sum_{\substack{m_0\leq l<\\ \left|k_{j}
(\sigma)\right|+\rho_{j}}} p\epsilon(\sigma)c_
{j,l}(\sigma)\cdot\frac{1}{l+s}
+\sum_{\eta(\sigma)}\frac{\epsilon(\sigma)}{2}\left(\frac{1}{s-\eta(\sigma)}
+\frac{1
}{s-\bar{\eta}(\sigma)} \right)\\ &-\sum_{
\beta(\sigma)}\frac{
\epsilon(\sigma)}{s+\beta(\sigma)}-\frac{\epsilon(\sigma)\log q(\sigma)}{2}
+K(\sigma,s) +sC(\sigma;s),
\end{align*}
where all involved sums converge absolutely.
\end{prop}
\begin{proof}
We let $s_{1}=:s$ and fix
$s_{2},\dots,s_{N}$. If we multiply \eqref{EqPrTr} by 
$2s\prod_{i'=2}^{N}\left(s_{i'}^{2}-s^{2}\right)$ 
and use that for $\sigma=w_{0}\sigma$ we have
$k_{n+1}(\sigma)=0$ and $c_{1}(\sigma)+c_{2}(\sigma)=p\dim(\sigma)$,
the Lemma follows from Propositon \ref{Resolvent} and from the above
computations. 
\end{proof}

\section{Twisted Dirac Operators}\label{sectwDO}
\setcounter{equation}{0}
In this section we introduce certain twisted Dirac operators on $\widetilde X$
and
compute
the Fourier transform of the corresponding heat-kernel. We keep the notations of
section \ref{Subsecbr}. Let $\Cl(\mathfrak{p})$
be the Clifford-algebra of $\mathfrak{p}$ with
respect to the scalar product on
$\mathfrak{p}$ defined in section \ref{subsH}. Let
$\Cl(\widetilde X):=G\times_{\Ad}
\Cl(\mathfrak{p})$ be the
Clifford bundle over $\widetilde X$. Let
$\tilde{S}=G\times_{\kappa}\Delta^{2n}$ be the spinor bundle of $\widetilde X$.
We denote by $\cc: \Cl(\mathfrak{p})\otimes \Delta^{2n}\longrightarrow
\Delta^{2n},\: X\otimes
v\mapsto \cc(X)v$ the Clifford multiplication. Since $M$ centralizes
$\mathfrak{a}$,
there is an $\epsilon\in\{\pm 1\}$ such that $\epsilon \cc(H_{1})$ acts on the
spaces
$\Delta^{2n}_{\pm}$ with eigenvalues $\mp i$, where $H_1$ is in section
\ref{subsH}.\\
Now let $\sigma\in\hat{M}$ and assume that $\sigma\neq w_0\sigma$,
$k_{n+1}(\sigma)>0$.
Let $\nu(\sigma)\in\hat{K}$ be as in Proposition \ref{branching}. Let
$\tilde{E}(\sigma)$
be the
vector bundle over $\widetilde X$ as in section \ref{secsymzeta}. Then by
Proposition \ref{branching} one has
$\tilde{E}(\sigma)=\tilde{E}_{\nu(\sigma)}\otimes \tilde{S}$. We denote by 
$\tilde{D}(\sigma)$ the twisted Dirac-Operator on $\tilde{E}(\sigma)$,
multiplied 
by $\epsilon$.  Then we have the following proposition.
\begin{prop}\label{DsqA}
Let $\tilde{A}(\sigma)$ be as in section \ref{secsymzeta}. Then one has
$\tilde{A}(\sigma)=\tilde{D}(\sigma)^2$.
\end{prop}
\begin{proof}
We regard $\tilde{D}(\sigma)$ as an operator on
$C^\infty(G,\nu(\sigma)\otimes\kappa)$, where the latter 
space is as in \eqref{defsect}. In \cite{AS}, it was
assumed that $\rk(G)=\rk(K)$. However,
one can proceed exactly as in \cite[page 53-54]{AS} to obtain
$\tilde{D}
(\sigma)^2=-\Omega+\left(\nu(\sigma)(\Omega_K)-\kappa(\Omega_K)\right)\Id$, 
which is analogous to \cite[equation (A 9)]{AS}.
The highest weight
$\Lambda(\nu(\sigma))$ of $\nu(\sigma)$ is as in Proposition \ref{branching} 
and the highest weight of $\kappa$ is given by
$\Lambda(\kappa)=\frac{1}{2}e_2+\ldots+\frac{1}{2}e_{n+1}$.  
Thus applying \cite[equation 2.18]{MP1} 
one computes $\nu(\sigma)\left(\Omega_K\right)-\kappa(\Omega_K)=c(\sigma)$ and
the Lemma is proved.
\end{proof}
Then $\tilde{D}(\sigma)$ and
$\tilde{D}(\sigma)^2$ with domain
$C_c^\infty(G,\nu(\sigma)\otimes\kappa)\subset L^2(G,\nu(\sigma)\otimes\kappa)$
are essentially selfadjoint, see \cite[Theorem II.5.7]{LM} and \cite[section
4]{MP}. Their selfajoint closures will be denoted by the
same symbols.
 Let $X_1,\dots,X_{2n+1}$
be an orthonormal base of $\mathfrak{p}$. Then by \eqref{Eqheat}
$\tilde{D}(\sigma)e^{-t\tilde{D}(\sigma)^2}$ acts on
$L^2(G,\nu(\sigma)\otimes\kappa)$ as a convolution operator with smooth kernel
\begin{align*}
K_t^\sigma(g):=\epsilon e^{-tc(\sigma)}\sum_{i=1}^{2n+1}\Id\otimes
\cc(X_i)\circ\frac{d}{dt}\biggr|_{t=0}
H_t^{\nu(\sigma)\otimes\kappa}(\exp{-tX_i}g),
\end{align*}
where $H_t^{\nu(\sigma)\otimes\kappa}$ is as in \eqref{DefH}.
Since $H_t^{\nu(\sigma)\otimes\kappa}$ belongs to all 
Harish-Chandra Schwarz spaces 
$\left(\mathcal{C}^{q}(G)\otimes
\End\left(V_{\nu(\sigma)}\otimes\Delta^{2n}\right)\right)^K$, $q>0$ and since
these
spaces are stable under the differential 
action of $U(\mathfrak{g}_\C)$, it follows that
$K_t^{\nu(\sigma)}\in\left(\mathcal{C}^{q}(G)\otimes
\End\left(V_{\nu(\sigma)}\otimes\Delta^{2n}\right)\right)^K$.
Define a $K$-finite Schwarz function $k_t^\sigma:=\Tr K_t^\sigma$,
where $\Tr$ is the trace in
$\End\left(V_{\nu(\sigma)}\otimes\Delta^{2n}\right)$. 
Then the Fourier transform of $k_t^\sigma$ was computed by Moscovici and Stanton
\cite{MS}. It is given as follows.

\begin{prop}\label{FTDir}
Let $\sigma\in\hat{M}$, $k_{n+1}(\sigma)>0$.
Then for $\lambda\in\R$ one has
\begin{align*}
&\Theta_{\sigma,\lambda}(k_t^\sigma)=(-1)^{n}\lambda
e^{-t\lambda^{2}};\quad
\Theta_{w_{0}\sigma,\lambda}(k_t^\sigma)=(-1)^{n+1}\lambda
e^{-t\lambda^{2}}.
\end{align*}
Moreover, if $\sigma'\in\hat{M}$, $\sigma'\notin\{\sigma,w_{0}\sigma\}$, for
every
$\lambda\in\R$ one has $\Theta_{\sigma',\lambda}(k_t^\sigma)=0$.
\end{prop}
\begin{proof}
Let $\pi$ be an admissible unitary representation of $G$ on $\mathcal{H}_\pi$.
Define an
operator $\tilde{D}_{\sigma}(\pi)$ on the $K$-invariant subspace
$\left(\mathcal{H}_{\pi}\otimes
V_{\nu(\sigma)}\otimes\Delta^{2n}\right)^{K}$ 
by $\tilde{D}_{\sigma}(\pi):=\epsilon\sum_{i=1}^{2n+1}\pi(X_{i}
)\otimes\Id\otimes\cc(X_i)$.
Then combining the arguments of \cite{Barbasch}, \cite[section 4]{MP} and
Proposition \ref{DsqA} one
easily obtains $\Tr\left(\pi(k_t^\sigma)\right)=e^{t(\pi(\Omega)-c(\sigma))}
\Tr\tilde{D}_{\sigma}(\pi)$. Now by \cite[Proposition 3.6]{MS}, where an $i$ has
to be added by
the arguments
of Chapter 3 in \cite{MS}, it follows that for every $\sigma'\in\hat{M}$ one
has $\Tr\tilde{D}_{\sigma}(\pi_{\sigma',\lambda})=\lambda\left(\dim{\left(V_{
\sigma'}\otimes V_{\nu(\sigma)}\otimes\Delta^{2n}_{+}\right)}^{M}
-\dim{\left(V_\sigma'\otimes
V_{\nu(\sigma)}\otimes\Delta^{2n}_{-}\right)}^{M}\right)$.
Let $\check{\sigma}$ be the contragredient representation of $\sigma$. By
\cite[section 3.2.5]{Goodman} if $n$ is
even one has $\check{\sigma}=\sigma$ 
and if $n$ is odd one has
$\check{\sigma}=w_0\sigma$. Hence by Proposition
\ref{branching}, for $\sigma'\in\hat{M}$ and $\lambda\in\R$ one has
$\dim{\left(V_{\sigma'}\otimes
V_{\nu(\sigma)}\otimes\Delta^{2n}_{+}\right)}^{M}
-\dim{\left(V_{\sigma'}\otimes
V_{\nu(\sigma)}\otimes\Delta^{2n}_{-}\right)}^{M}=(-1)^{n}\left[
\sigma-w_0\sigma:\sigma'\right]$. If we apply the formula for
$\pi_{\sigma,\lambda}(\Omega)$ from section \ref{subsecprs} , the
Proposition follows.
\end{proof}

\section{The antisymmetric Selberg zeta function}\label{The antisymmetric
Selberg zeta function}\label{secasyms}
\setcounter{equation}{0}
Let $\sigma\in\hat{M}$ with $k_{n+1}(\sigma)>0$. Let
$E(\sigma)=\Gamma\backslash\tilde{E}(\sigma)$
be the locally homogeneous vector bundle over $X$ as in section
\ref{secsymzeta}. Let $\tilde{D}(\sigma)$ be as in the previous section. 
Then $\tilde{D}(\sigma)$ pushes down to an operator $D(\sigma)$ on the
smooth sections of $E(\sigma)$. By Proposition \ref{DsqA} one has
$D(\sigma)^2=A(\sigma)$,
where $A(\sigma)$ is the operator from section \ref{secsymzeta}. The operator
$D(\sigma)$ 
has a well defined restriction to
$L^{2}_{d}(\Gamma\backslash G,\nu(\sigma)\otimes\kappa)$ in the sense of
unbounded operators.
This restriction will be denoted by $D(\sigma)_d$. The space
$L^{2}_{d}(\Gamma\backslash G,\nu(\sigma)\otimes\kappa)$
is an orthogonal direct sum of finite dimensional eigenspaces of
$D(\sigma)_d^{2}$.
On these spaces $D(\sigma)_d$ acts as a symmetric operator, and
thus
$L^{2}_{d}(\Gamma\backslash G,\nu(\sigma)\otimes\kappa)$ is also the orthogonal
direct sum
of eigenspaces of
$D(\sigma)_d$.
By \eqref{Wlaw} the operator
$D(\sigma)_d\:e^{-tD(\sigma)_d^{2}}$ is of trace class. One has
\begin{align}\label{TrDO}
\Tr\left(
D(\sigma)_d\:e^{-tD(\sigma)_{d}^{2}}\right)=\Tr\left(\pi_{\Gamma,d}(k_t^
\sigma)\right),
\end{align}
where $k_t^\sigma$ is as in the previous section.  Let $\{\mu_k\}$ be a fixed
sequence of the eigenvalues of
$D(\sigma)_d$
with
$\mu_{k}=\mu_{k'}$ iff $k=k'$. This sequence might be finite or infinite. Let 
$\mathcal{E}(\mu_{k})$ be the eigenspace of $D(\sigma)$ corresponding to the
eigenvalue $\mu_{k}$ and let
$d(\mu_{k},\sigma):=\dim(\mathcal{E}(\mu_{k}))$. Let $N\in\N$, $N>n+1$. We
choose
distinct points $s_{1},\dots,s_{N}$ such that $\Real(s_j)>2n$, $\Real(s_j^2)>0$
for all j. By Proposition \ref{Resolvent} we have
\begin{align}\label{GlSpII}
2i\sum_{k}d(\mu_{k},\sigma)\mu_{k}\prod_{j=1}^{N}\frac{1}{\mu_{k}^{2}+s_{j}^{2}}
&=2i\int_{0}^{
\infty}\sum_{j=1}^{N}e^{-ts_j^2}\Tr\left(D(\sigma)_{d}\:e^{-tD(\sigma
)_{d}^{2}}\right)\prod_{\substack{j'=1\\ j'\neq
j}}^{N}\frac{1}{s_{j'}^{2}-s_{j}^{2}}dt.
\end{align}
The sum on the left hand side converges 
absoluetly by \eqref{Wlaw}.\\
We compute the right hand side of \eqref{GlSpII} using \eqref{TrDO} and the
invariant trace formula stated in Theorem \ref{Spurf}.
By \cite[equation (2.22)]{MP} we have
$P_{\sigma}(\lambda)=P_{w_{0}\sigma}(\lambda)$. By Proposition \ref{Para} we
have
$\Omega(\sigma,\lambda)=\Omega(w_{0}\sigma,\lambda)$. 
Thus, applying \eqref{Idcontr}, \eqref{Fouriertrafo T} \eqref{Resfrml}, Theorem
\ref{Hoffmanns
Theorem} and Proposition \ref{FTDir}, it follows that 
$I(k_t^\sigma)$, $T(k_t^\sigma)$, $\mathcal{I}(k_t^\sigma)$ and $R(k_t^\sigma)$
vanish.\\
If $L(\gamma,\sigma)$, $L(\gamma,w_0\sigma)$ are as in \eqref{hyperbcontr}, it
follows from
\eqref{Hyperb} and Proposition
\ref{FTDir} that
\begin{align*}
&2i\int_{0}^{\infty}\sum_{j=1}^{N}e^{-ts_{j}^{2}}H(k_t^\sigma
)\prod_{\substack{j'=1\\ j'\neq
j}}^{N}\frac{1}{s_{j'}^{2}-s_{j}^{2}}dt\\
&=(-1)^{n}\sum_{j=1}^{N}\sum_{\left[\gamma\right]\in\CC(\Gamma)_{\s
}-\left[1\right]}l(\gamma_{0})(L(\gamma,
\sigma)-L(\gamma,w_{0}\sigma))e^{-l(\gamma)s_{j}}\prod_{\substack{j'=1\\ j'\neq
j}}^{N}\frac{1}{s_{j'}^{2}-s_{j}^{2}}.
\end{align*}
For $\Real(s)>2n$ we define the anti-symmetric Selberg-zeta function by
$S_{a}(s,\sigma):=\frac{Z(s,\sigma)}{Z(s,w_{0}\sigma)}$. As above, by
\cite[section 3.2.5]{Goodman} if $n$ is
even we have $\overline{\Tr(\sigma)}=\Tr(\sigma)$ 
and if $n$ is odd we have
$\overline{\Tr(\sigma)}=\Tr(w_0\sigma)$ for every $\sigma\in\hat{M}$.
Thus using \eqref{LogSZF} 
we obtain
\begin{align*}
2i\int_{0}^{\infty}\sum_{j=1}^{N}e^{-ts_{j}^{2}}H(k_t^\sigma
)\prod_{\substack{j'=1\\ j'\neq
j}}^{N}\frac{1}{s_{j'}^{2}-s_{j}^{2}}dt=\sum_{j=1}^{N}\frac{d}{ds}\biggr|_{s=s_j
}\log{S_{a}
(s,\sigma)}\prod_{\substack{j'=1\\ j'\neq
j}}^{N}\frac{1}{s_{j'}^{2}-s_{j}^{2}}.
\end{align*}
Now we treat the distribution $\mathcal{S}$. Let $\eta(\sigma)$,
$\eta(w_0\sigma)$ be the 
poles of $\det{\left(I(\sigma)\circ\mathbf{C}(\nu_\sigma:\sigma:s)\right)}$
resp. of $\det{\left(I(w_0\sigma)\circ\mathbf{C}(\nu_\sigma:w_0\sigma
:s)\right)}$ with negative real part and with multiplicity by
$\dim(\nu_\sigma)$ as in section \ref{secC}. Then by \eqref{Sc}, \eqref{SpC},
Remark \ref{RmrkC} and Proposition \ref{FTDir} one has
\begin{align*}
2i\int_{0}^{\infty}\sum_{j=1}^{N}e^{-ts_{j}^{2}}S(k
_t^\sigma)\prod_{\substack{j'=1\\ j'\neq
j}}^{N}\frac{1}{s_{j'}^{2}-s_{j}^{2}}dt=&(-1)^{n+1}\sum_{\eta(\sigma)}
\sum_{j=1}^{N}\frac{1}{
s_j-
\eta(\sigma)}\prod_{\substack{j'=1\\
j'\neq
j}}^{N}\frac{1}{s_{j'}^{2}-s_{j}^{2}}\\
&+(-1)^{n}\sum_{
\eta(w_0\sigma)}\sum_{j=1
}^{N}\frac{1}{s_j-
\eta
(w_0\sigma)}\prod_{\substack{
j'=1\\ j'\neq
j}}^{N}\frac{1}{s_{j'}^{2}-s_{j}^{2}}.
\end{align*}
Let $\sigma\in\hat{M}$ with $k_{n+1}(\sigma)<0$. Then we define the twisted
Dirac operator $D(\sigma)$ as $D(\sigma):=-D(w_0\sigma)$, where $D(w_0\sigma)$
is as in 
section \ref{sectwDO}. We can now prove our main result about the antisymmetric
Selberg zeta
function.

\begin{prop}\label{PropAsymS}
The antisymmetric Selberg zeta function $S_a(s,\sigma)$ satisfies 
\begin{align*}
\frac{S_{a}'(s,\sigma)}{S_{a}(s,\sigma)}=&\sum_{k}\left(\frac{d(\mu_k,\sigma)}{
s-i\mu_k
}-\frac{d(\mu_k,\sigma)}{s+i\mu_k}+C(\mu_k;s)\right)
+(-1)^{n}\sum_{\eta(\sigma)}\left(\frac{1}{s-\eta(\sigma)}
+C'(\eta(\sigma);s\right)\\
&+(-1)^{n+1}\sum_{\eta(w_0\sigma)}\left(\frac{1}{s-\eta(w_0\sigma)}
+C'(\eta(w_0\sigma);s)\right),
\end{align*}
where $C(\mu_k;s)$, $C(\eta(\sigma);s)$ and $C(\eta(w_0\sigma);s)$ are
polynomials in $s$ of degree at most $2(N-2)$.
\end{prop}
\begin{proof}
Clearly we can assume that $k_{n+1}(\sigma)>0$. We let $s_{1}=:s$,
$\Real(s)>2n$, $\Real(s^2)>0$ and fix distinct
$s_{2},\dots,s_{N}$
with
$\Real(s_j)>2n$, $\Real(s_j^2)>0$. Then we multiply both sides of
\eqref{GlSpII} by 
$\prod_{j'=2}^{N}\left(s_{j'}^{2}-s^{2}\right)$
and apply Proposition \ref{Resolvent} to the left hand side and equation
\eqref{TrDO} and
Theorem \ref{Spurf} to the right hand side.
The proposition follows immediately from the previous computations.
\end{proof}
Now we can complete our proof of the meromorphic
continuation of the Selberg
zeta function $Z(s,\sigma)$ for every $\sigma\in\hat{M}$ and describe its
singularities. We say that a
meromorphic
function
$f\colon\C\to\C$ has a singularity of order $k\in\Z$ at $s_0\in\C$ if 
$\lim_{s\to s_0}(s-s_0)^{-k}f(s)$ exists and is not zero.
\begin{thrm}\label{AnconctZ}
Let $\sigma\in\hat{M}$. Then the Selberg zeta function $Z(s,\sigma)$
has a meromorphic continuation to $\mathbb{C}$. Its
singularities associated 
to spectral parameters are as follows:
\begin{enumerate}
\item If $\sigma=w_0\sigma$ at the points $\pm i\sqrt{\lambda_k}$ of order
$m_{\s}(\lambda_{k},\sigma)$, where $\lambda_k$ is a non-zero eigenvalue of
$A(\sigma)$ and 
$m_{\s}(\lambda_{k},\sigma)$ is the graded dimension of the corresponding
eigenspace. Here for $\Real(\lambda_k)<0$ we take the square root with positive
imaginary part.
\item If $\sigma\neq w_0\sigma$ at the points $\pm i\mu_k$ of order
$\frac{1}{2}\left(m_{\s}(\mu_k^2,\sigma)+d(\pm\mu_k,\sigma)-d(\mp\mu_k,
\sigma)\right)$. Here $\mu_k$ is a non-zero eigenvalue of
$D(\sigma)$ of multiplicity $d(\mu_k,\sigma)$ and $m_{\s}(\mu_k^2,\sigma)$ is
the graded 
dimension of the eigenspace of $A(\sigma)$ corresponding to the eigenvalue
$\mu_k^2$.
\item At the point $s=0$ of order $2m_{\s}(0,\sigma)-c_1(\sigma)$
if $\sigma=w_0\sigma$ and of order $m_{\s}(0,\sigma)$ if $\sigma\neq
w_0\sigma$. Here $c_1(\sigma)$ is as in \eqref{Resfrml}.
\item At the points $-\beta(\sigma)$ where $\beta(\sigma)$ are the poles of
$\det{\left(I(\sigma)\circ
\mathbf{C}(\nu_\sigma:\sigma:s) \right)}$ on $(0,n]$. The order at
$\beta(\sigma)$ is
$-m(\beta(\sigma))$, where $m(\beta(\sigma))$ is the corresponding multiplicity
divided by
$\dim(\nu_\sigma)$ as in section \ref{secC}. 
\item If $n$ is even at the points $\eta(\sigma)$ of order $m(\eta(\sigma))$  
and if $n$ is odd at the points $\eta(w_0\sigma)$ of order
$m(\eta(w_0\sigma))$. Here
$\eta(\sigma)$ resp. $\eta(w_0\sigma)$ are the poles of
$\det{\left(I(\sigma)\circ
\mathbf{C}(\nu_\sigma:\sigma:s) \right)}$ resp. $\det{\left(I(\sigma)\circ
\mathbf{C}(\nu_\sigma:\sigma:s) \right)}$ with negative real part
and $m(\eta(\sigma)), m(\eta(w_0\sigma)$ are the corresponding multiplicities
divided by
$\dim(\nu_\sigma)$ as in section \ref{secC}.
\end{enumerate}
Moreover, the Selberg zeta function $Z(s,\sigma)$ has 
singularities which depend
on $\Gamma$ only via $p$, the number of cusps of $\Gamma$. If all $k_j(\sigma)$
are integers, they are
located at the negative integers and if all $k_j(\sigma)$ are half integers,
they are located at the
negative half integers.  \\
Explicitly, if all $k_j(\sigma)$ are integers, they are given as follows.
\begin{enumerate}
\item At the points $-l$, $l\in\N$, $l\geq m_0$ of  order
$-p\dim(\sigma)$.
\item At the points $-l$, $l\in\N$, $m_0\leq l<
\left|k_{j}(\sigma)\right|+\rho_{j}$,
$j=2,\dots,n+1$ of order $pc_
{j,l}(\sigma)$.
\end{enumerate}
Here $m_0$ and the $c_{j,l}(\sigma)$ are as in Proposition \ref{Para}.
If all $k_j(\sigma)$ are half integers, they can be described in the same way if
$\N$
is replaced by $\frac{1}{2}\N-\N$. \\
The above enumeration exhausts all possible singularities of non-zero order of
$Z(s,\sigma)$ if at overlapping singularities the orders are added up.
\end{thrm}
\begin{proof}
Assume that all $k_j(\sigma)$ are integers. If $\sigma=w_0\sigma$, the Theorem
follows 
from Remark \ref{RmrkC} and Proposition \ref{PropSymS}.
It remains to consider the case that $\sigma\neq w_0\sigma$. By Proposition
\ref{DsqA}, the eigenvalues $\lambda_k$ of $A(\sigma)$ are given 
by the $\mu_k^2$, where $\mu_k$ are the eigenvalues of $D(\sigma)$. Moreover one
has $Z'(s,\sigma)/Z(s,\sigma)=\frac{1}{2}(S'(s,
\sigma)/S(s,\sigma)+S_{a}'(s,\sigma)/S_{a}(s,\sigma))$. Thus, combining
Proposition 
\ref{PropSymS}
and Proposition \ref{PropAsymS}, it follows that
the logarithmic derivative of $Z(s,\sigma)$ has a meromorphic continuation to
$\C$ and that its poles with the corresonding residues are as in the
enumeration. It 
remains to show that its residues are integral. 
For every $k$ let
$\mathcal{E}(\pm\mu_k)$
be the eigenspaces of $D(\sigma)$ corresponding to the eigenvalue  $\pm\mu_k$ 
and let $\mathcal{E}(\mu_k^2)$ be the eigenspace of $A(\sigma)$ corresponding to
the eigenvalue $\mu_k^2$.
Then one has
$\mathcal{E}(\mu_k^2)=\mathcal{E}(\mu_k)\oplus\mathcal{E}(-\mu_k)$.
Recall that $m_{\s}(\mu_k^2,\sigma)=\dim_{\gr}\mathcal{E}(\mu_k^2)$. Thus
$m_{\s}(\mu_k^2,\sigma)+d(\mu_k,\sigma)-d(-\mu_k,\sigma)$
is even. If all $k_j(\sigma)$ are half-integral, one proceeds in the 
same way by making the appropriate modifications in Proposition \ref{Para} resp.
Proposition \ref{PropSymS}. 
\end{proof}

\end{document}